\documentclass{amsart}

\usepackage{amscd}
\usepackage{amssymb}
\usepackage[T1]{fontenc}
\usepackage[utf8]{inputenc}
\usepackage{hyperref}
\usepackage[arrow,curve,matrix]{xy}
\usepackage{times}

\DeclareFontFamily{OMS}{rsfs}{\skewchar\font'60}
\DeclareFontShape{OMS}{rsfs}{m}{n}{<-5>rsfs5 <5-7>rsfs7 <7->rsfs10 }{}
\DeclareSymbolFont{rsfs}{OMS}{rsfs}{m}{n}
\DeclareSymbolFontAlphabet{\scr}{rsfs}

\renewcommand{\P}{\mathbb{P}} 
\renewcommand{\O}{\sO}
\newcommand{\Q}{\mathbb{Q}}

\newcommand\resto[1]{\hbox{\hbox{$\big\vert{}_{_{#1}}$}}}
 
\newcommand{\ratmap}{\dasharrow}

\newcommand{\sA}{\scr{A}}
\newcommand{\sB}{\scr{B}}
\newcommand{\sC}{\scr{C}}

\newcommand{\sE}{\scr{E}}
\newcommand{\sF}{\scr{F}}
\newcommand{\sG}{\scr{G}}
\newcommand{\sH}{\scr{H}}

\newcommand{\sK}{\scr{K}}

\newcommand{\sO}{\scr{O}}

\newcommand{\sT}{\scr{T}}

\newcommand{\bA}{\mathbb{A}}

\newcommand{\bC}{\mathbb{C}}

\newcommand{\bN}{\mathbb{N}}

\newcommand{\bQ}{\mathbb{Q}}

\DeclareMathOperator{\Exc}{Exc}

\DeclareMathOperator{\rk}{rk}

\DeclareMathOperator{\codim}{codim}

\DeclareMathOperator{\Hilb}{Hilb}

\DeclareMathOperator{\reg}{reg}
\DeclareMathOperator{\sing}{sing}
\DeclareMathOperator{\Sym}{Sym}
\DeclareMathOperator{\supp}{supp}

\DeclareMathOperator{\Isom}{Isom}
\DeclareMathOperator{\Var}{Var}

\DeclareMathOperator{\Aut}{Aut}

\DeclareMathOperator{\length}{{length}}

\newcommand{\into}{\hookrightarrow}

\newcommand{\wt}{\widetilde}

\newcommand{\wtilde}{\widetilde}
\newcommand{\what}{\widehat}

\newcounter{thisthm}
\newcommand{\ilabel}[1]{\newcounter{#1}\setcounter{thisthm}{\value{thm}}\setcounter{#1}{\value{enumi}}}
\newcommand{\iref}[1]{(\thesection.\the\value{thisthm}.\the\value{#1})}

\theoremstyle{plain}    
\newtheorem{thm}{Theorem}[section]
\newtheorem{defn}[thm]{Definition}

\newtheorem{assumption}[thm]{Assumption} 
 
\numberwithin{equation}{thm}
\numberwithin{figure}{section}
\theoremstyle{plain}    
\newtheorem{cor}[thm]{Corollary}
\newtheorem{lem}[thm]{Lemma}
\newtheorem{conjecture}[thm]{Conjecture}

\theoremstyle{plain}    
\newtheorem{prop}[thm]{Proposition}
\newtheorem{proclaim-special}[thm]{\specialthmname}

\theoremstyle{remark}
\newtheorem{rem}[thm]{Remark}
\newtheorem{obs}[thm]{Observation}

\newtheorem{iass}[thm]{Assumption} 
\newtheorem{ihyp}[thm]{Induction Hypothesis}
\newtheorem{subrem}[equation]{Remark}
\newtheorem{claim}[thm]{Claim} 
\newtheorem*{claim*}{Claim} 
\newtheorem{notation}[thm]{Notation}

\newtheoremstyle{bozont-remark}{3pt}{3pt}%
     {}
     {}
     {\it}
     {.}
     {.5em}
     {\thmname{#1}\thmnumber{ #2}: \thmnote{\sc #3}}
\theoremstyle{bozont-remark}

\newtheoremstyle{bozont-reverse}{3pt}{3pt}%
     {}
     {}
     {\it}
     {.}
     {.5em}
     {\thmnumber{ #2} \thmname{#1}: \thmnote{\sc #3}}
\theoremstyle{bozont-reverse}    
\newtheorem{remark-special}[thm]{\specialremname}
\newenvironment{named_rem}[1]
     {\def\specialremname{#1}\begin{remark-special}}
     {\end{remark-special}}

\def\factor#1.#2.{\left. \raise 2pt\hbox{$#1$} \right/\hskip -2pt\raise
  -2pt\hbox{$#2$}} 
\newlength{\swidth}
\newenvironment{enumerate-p}{
  \begin{enumerate}}
  {\setcounter{equation}{\value{enumi}}\end{enumerate}}

\date{\today}

\author{Stefan Kebekus}
\author{S\'andor J.\ Kov\'acs}

\thanks{Stefan Kebekus was supported in part by the DFG-Forschergruppe
  ``Classification of Algebraic Surfaces and Compact Complex
  Manifolds''.  S\'andor Kov\'acs was supported in part by NSF Grant
  DMS-0554697 and the Craig McKibben and Sarah Merner Endowed
  Professorship in Mathematics.}
  
\address{Stefan Kebekus, Mathematisches Institut, Albert-Ludwigs-Universit\"at
  Freiburg, Eckerstraße 1, 79104 Freiburg, Germany}
\email{\href{mailto:stefan.kebekus@math.uni-freiburg.de}{stefan.kebekus@math.uni-freiburg.de}}
\urladdr{\href{http://home.mathematik.uni-freiburg.de/kebekus}{http://home.mathematik.uni-freiburg.de/kebekus}}

\address{\noindent S\'andor Kov\'acs, University of Washington,
  Department of Mathematics, Box 354350, Seattle, WA 98195, U.S.A.}
\email{\href{mailto:kovacs@math.washington.edu}{kovacs@math.washington.edu}}
\urladdr{\href{http://www.math.washington.edu/~kovacs}{http://www.math.washington.edu/$\sim$kovacs}}

\title[Surfaces and threefolds mapping to the moduli stack]{The
  structure of surfaces and threefolds mapping to the moduli stack of
  canonically polarized varieties}

\begin{document}

\begin{abstract}
  Generalizing the well-known Shafarevich hyperbolicity conjecture, it has
  been conjectured by Viehweg that a quasi-projective manifold that admits a
  generically finite morphism to the moduli stack of canonically polarized
  varieties is necessarily of log general type.  Given a quasi-projective
  threefold $Y^\circ$ that admits a non-constant map to the moduli stack, we
  employ extension properties of logarithmic pluri-forms to establish a strong
  relationship between the moduli map and the minimal model program of
  $Y^\circ$: in all relevant cases the minimal model program leads to a fiber
  space whose fibration factors the moduli map. A much refined affirmative
  answer to Viehweg's conjecture for families over threefolds follows as a
  corollary.  For families over surfaces, the moduli map can be often be
  described quite explicitly.  Slightly weaker results are obtained for
  families of varieties with trivial, or more generally semi-ample canonical
  bundle.
\end{abstract}

\maketitle
\tableofcontents

\section{Introduction and main results}

\subsection{Introduction}

Let $Y^\circ$ be a quasi-projective manifold that admits a morphism $\mu:
Y^\circ \to \mathfrak M$ to the moduli stack of canonically polarized
varieties. Generalizing the classical Shafarevich hyperbolicity conjecture
\cite{Shaf63}, Viehweg conjectured in \cite[6.3]{Viehweg01} that $Y^\circ$ is
necessarily of log general type if $\mu$ is generically finite. Equivalently,
if $f^\circ: X^\circ \to Y^\circ$ is a smooth family of canonically polarized
varieties, then $Y^\circ$ is of log general type if the variation of $f^\circ$
is maximal, i.e., $\Var(f^\circ) = \dim Y^\circ$.  We refer to \cite{KK08} for
the relevant notions, for detailed references, and for a brief history of the
problem, but see also \cite{KS06}.

Viehweg's conjecture was confirmed for $2$-dimensional manifolds $Y^\circ$ in
\cite{KK08} using explicit surface geometry. Here, we employ recent extension
theorems for logarithmic forms to study families over threefolds. If $\dim
Y^\circ \leq 3$, we establish a strong relationship between the moduli map
$\mu$ and the logarithmic minimal model program of $Y^\circ$: in all relevant
cases, any logarithmic minimal model program necessarily terminates with a
fiber space whose fibration factors the moduli map.  This allows us to prove a
much refined version of Viehweg's conjecture for families over surfaces and
threefolds, and give a positive answer to the conjecture even for families of
varieties with only semi-ample canonical bundle. If $Y^\circ$ is a surface we
recover the results of \cite{KK08} in a more sophisticated manner. In fact,
going far beyond those results we give a complete geometric description of the
moduli map in those cases when the variation cannot be maximal.

The proof of our main result is rather conceptual and independent of the
argumentation of \cite{KK08} which essentially relied on combinatorial
arguments for curve arrangements on surfaces and on Keel-McKernan's solution
to the Miyanishi conjecture in dimension 2, \cite{KMcK}. Many of the
techniques introduced here generalize well to higher dimensions, most others
at least conjecturally.

Throughout the present paper we work over the field of complex numbers.

\subsection{Main results}

The main results of the present paper are summarized in the following theorems
which describe the geometry of families over threefolds under increasingly
strong hypothesis.

\begin{thm}[Viehweg conjecture for families over threefolds]\label{thm:mainresult0}
  Let $f^\circ: X^\circ \to Y^\circ$ be a smooth projective family of
  varieties with semi-ample canonical bundle, over a quasi-projective manifold
  $Y^\circ$ of dimension $\dim Y^\circ \leq 3$. If $f^\circ$ has maximal
  variation, then $Y^\circ$ is of log general type. In other words,
  $$
  \Var(f^\circ) = \dim Y^\circ \, \Rightarrow \, \kappa(Y^\circ) = \dim
  Y^\circ.
  $$
\end{thm}

\begin{subrem}
  The definition of Kodaira dimension $\kappa(Y^\circ)$ for quasi-projective
  manifolds is recalled in Notation~\ref{not:Kodaira} below.
\end{subrem}

For families of \emph{canonically} polarized varieties, we can say much
more. The following much stronger theorem gives an explicit geometric
explanation of Theorem~\ref{thm:mainresult0}.

\begin{thm}[Relationship between the moduli map and the MMP]\label{thm:mainresult2}
  Let $f^\circ: X^\circ \to Y^\circ$ be a smooth projective family of
  canonically polarized varieties, over a quasi-projective manifold $Y^\circ$
  of dimension $\dim Y^\circ \leq 3$. Let $Y$ be a smooth compactification of
  $Y^\circ$ such that $D := Y \setminus Y^\circ$ is a divisor with simple
  normal crossings.

  Then any run of the minimal model program of the pair $(Y,D)$ will terminate
  in a Kodaira or Mori fiber space whose fibration factors the moduli map
  birationally.
\end{thm}
\begin{subrem}
  If $\kappa(Y^\circ) = 0$ in the setup of Theorem~\ref{thm:mainresult2}, then
  any run of the minimal model program will terminate in a Kodaira fiber space
  that maps to a single point. Since this map to a point factors the moduli
  map birationally, Theorem~\ref{thm:mainresult2} asserts that the family
  $f^\circ$ is necessarily isotrivial if $\kappa(Y^\circ) = 0$.
\end{subrem}
\begin{subrem}
  Neither the compactification $Y$ nor the minimal model program discussed in
  Theorem~\ref{thm:mainresult2} is unique. When running the minimal model
  program, one often needs to choose the extremal ray that is to be
  contracted.
\end{subrem}

In the setup of Theorem~\ref{thm:mainresult2}, if $\kappa(Y^\circ) \geq 0$,
then the minimal model program terminates in a Kodaira fiber space whose base
has dimension $\kappa(Y^\circ)$. The following refined version of Viehweg's
conjecture is therefore an immediate corollary of
Theorem~\ref{thm:mainresult2}.

\begin{cor}[Refined Viehweg conjecture for families over threefolds cf.\
  \protect{\cite[1.6]{KK08}}]\label{cor:mainresult1}
  Let $f^\circ: X^\circ \to Y^\circ$ be a smooth projective family of
  canonically polarized varieties, over a quasi-projective manifold $Y^\circ$
  of dimension $\dim Y^\circ \leq 3$. Then either
  \begin{itemize}
  \item[i)] $\kappa(Y^\circ) = -\infty$ and $\Var(f^\circ) < \dim Y^\circ$, or
  \item[ii)] $\kappa(Y^\circ) \geq 0$ and $\Var(f^\circ) \leq
    \kappa(Y^\circ)$. \qed
  \end{itemize}
\end{cor}

As a further application of Theorem~\ref{thm:mainresult2}, we describe the
family $f^\circ: X^\circ \to Y^\circ$ explicitly if the base manifold
$Y^\circ$ is a surface and the variation is not maximal.

\begin{thm}[Description of the family in case of $\Var(f^\circ)=1$]\label{thm:mainresult3}
  Let $f^\circ: X^\circ \to Y^\circ$ be a smooth projective family of
  canonically polarized varieties, over a quasi-projective manifold $Y^\circ$
  of dimension $\dim Y^\circ=2$.  If $\kappa(Y^\circ) < 2$ and
  $\Var({f^\circ}) = 1$, then one of the following holds.
  \begin{enumerate}
  \item $\kappa({Y^\circ})=1$, and there exists an open set $U \subseteq
    Y^\circ$ and a Cartesian diagram of one of the following two types, %
    \vskip-2.2em
    $$
    \hskip1.3cm
    \xymatrix{%
      {\widetilde {U}} \ar[rr]^{\gamma}_{\text{\'etale}}
      \ar[d]_{\txt{\scriptsize $\widetilde \pi$\\\scriptsize elliptic
          \\ \scriptsize \quad fibration\quad\ }} && {{U}}
      \ar[d]^{\txt{\scriptsize $\pi$\vphantom{$\widetilde\pi$}\\
          \scriptsize elliptic \\
          \scriptsize \quad fibration\quad\ }} \\
      {\widetilde V} \ar[rr]_{\text{\'etale}} && {V}}
    \begin{minipage}[c]{.02\linewidth}
      \ \\
      \ \\
      \begin{center}
        \text{{or}}
      \end{center}
    \end{minipage}
    \xymatrix{%
      {\widetilde {U}} \ar[rr]^{\gamma}_{\text{\'etale}}
      \ar[d]_{\txt{\scriptsize
          $\widetilde \pi$\\ \scriptsize smooth\\
          \scriptsize algebraic\\ \scriptsize \quad
          $\bC^*$-bundle\quad\ }}
      && {{U}} \ar[d]^{\txt{\scriptsize $\pi\vphantom{\widetilde\pi}$\\
          \scriptsize smooth,  \\
          \scriptsize algebraic\\ \scriptsize \quad
          $\bC^*$-bundle\quad\ }}\\
      {\widetilde V} \ar@{=}[rr] && {V}}
    $$
    such that $f^\circ_{\wt U}: X^\circ \times_U \wtilde U\to \wt U$ is the
    pull-back of a family over $\widetilde V$.
  \item $\kappa({Y^\circ})= -\infty$, and there exists an open set $U
    \subseteq Y^\circ$ of the form $U = V \times \bA^1$ such that
    $X^\circ\resto U$ is the pull-back of a family over $V$.
  \end{enumerate}
\end{thm}

In order complete the description of families with non-maximal variation over
two-dimensional bases we include the following well-known statement. However,
we would like to point out that this is a much easier statement and follows by
simple abstract arguments, cf.~the proof of Lemma~\ref{lem:relglue}.

\begin{thm}[Description of the family in case of $\Var(f^\circ)=0$]
  Let $f^\circ: X^\circ \to Y^\circ$ be a smooth projective family of
  canonically polarized varieties, over a quasi-projective manifold
  $Y^\circ$. If $\Var({f^\circ}) = 0$, then there exists an open set $U
  \subseteq Y^\circ$ such that $X^\circ\resto U$ is isotrivial and further
  exists a finite \'etale cover $\wt U\to U$ such that $f^\circ_{\wt U}:
  X^\circ \times_U \wtilde U\to \wt U$ is trivial. \qed
\end{thm}

\subsection{Outline of proof, outline of paper}

The main results of this paper are shown in
Sections~\ref{sec:kinfty}--\ref{sec:kpos} where we consider the cases
$\kappa(Y^\circ) = -\infty$, $\kappa(Y^\circ) = 0$ and $\kappa(Y^\circ) > 0$
separately, the most difficult case being when $\kappa(Y^\circ) = 0$. To keep
the proofs readable, we have chosen to present many of the more technical
results separately in the preparatory
Sections~\ref{sec:notation}--\ref{sec:gluearama}. These may be of some
independent interest. The reader who is primarily interested in a broad
outline of the argument will likely want to take the technicalities on faith
and move directly to Sections~\ref{sec:kinfty}--\ref{sec:kpos} on the first
reading.

Section~\ref{sec:notation} introduces notation used in the remainder of the
present paper. In Section~\ref{sec:singularities}, we discuss certain classes
of singularities that appear in the minimal model program and recall the
Bogomolov vanishing result for log canonical threefolds. The standard
construction of the global index-one cover for good minimal models of Kodaira
dimension zero is recalled and summarized in Section~\ref{sec:indexcoverk0}.

Viehweg and Zuo have shown that the base of a family of positive variation
often carries an invertible sheaf of pluri-differentials whose Kodaira-Iitaka
dimension is at least the variation of the family. These \emph{Viehweg-Zuo
  sheaves}, which are central to our argumentation, are introduced and
discussed in Section~\ref{sec:VZ}.  The existence of a Viehweg-Zuo sheaf of
positive Kodaira-Iitaka dimension has strong consequences for the geometry if
the underlying space. These are discussed in Section~\ref{sec:VZ2}. We end the
preparatory part of the paper with Section~\ref{sec:gluearama} where we
discuss how families $f^\circ : X^\circ \to Y^\circ$ over a fibered base
$\pi^\circ : Y^\circ \to C^\circ$ that are isotrivial over $\pi^\circ$-fibers
often come from a family over $C^\circ$, at least after passing to an étale
cover.

\subsection{Acknowledgments}

We would like to thank Eckart Viehweg and Chengyang Xu for numerous
discussions that motivated the problem and helped to improve this paper.

 \part{TECHNIQUES}

\section{Notation and Conventions}
\label{sec:notation}

\subsection{Reflexive tensor operations}

When dealing with sheaves that are not necessarily locally free, we frequently
use square brackets to indicate taking the reflexive hull.

\begin{notation}[Reflexive tensor product]
  Let $Z$ be a normal variety and $\sA$ a coherent sheaf of
  $\O_Z$-modules. Given a number $n\in \bN$, set $\sA^{[n]} := (\sA^{\otimes
    n})^{**}$. If $\sA$ is reflexive of rank one, we say that $\sA$ is
  $\Q$-Cartier if there exists a number $n$ such that $\sA^{[n]}$ is
  invertible.
\end{notation}

We will later discuss the Kodaira dimension of singular pairs and the
Kodaira-Iitaka dimension of reflexive sheaves on normal spaces. Since this is
perhaps not quite standard, we recall the definition here.

\begin{notation}[Kodaira-Iitaka dimension of a sheaf]
  Let $Z$ be a normal projective variety and $\sA$ a reflexive sheaf of rank
  one on $Z$.  If $h^0\bigl(Z,\, \sA^{[n]}\bigr) = 0$ for all $n \in \mathbb
  N$, then we say that $\sA$ has Kodaira-Iitaka dimension $\kappa(\sA) :=
  -\infty$.  Otherwise, set
  $$
  M := \bigl\{ n\in \mathbb N \,|\, h^0\bigl(Z,\, \sA^{[n]}\bigr)>0\bigr\},
  $$
  recall that the restriction of $\sA$ to the smooth locus of $Z$ is locally
  free and consider the natural rational mapping
  $$
  \phi_n : Z \dasharrow \mathbb P\bigl(H^0\bigl(Z,\, \sA^{[n]}\bigr)^*\bigr)
  \text{ for each } n \in M.
  $$
  The Kodaira-Iitaka dimension of $\sA$ is then defined as
  $$
  \kappa(\sA) := \max_{n \in M} \bigl(\dim \overline{\phi_n(Z)}\bigr).
  $$
\end{notation}

\begin{notation}[Kodaira dimension of a quasi-projective variety]\label{not:Kodaira}
  If $Z^\circ$ is a quasi-projective manifold and $Z$ a smooth
  compactification such that $\Delta := Z \setminus Z^\circ$ is a divisor with
  at most simple normal crossings, define the Kodaira dimension of $Z^\circ$
  as $\kappa(Z^\circ) := \kappa \bigl(\sO_Z(K_Z+\Delta) \bigr)$. Recall the
  standard fact that this number is independent of the choice of the
  compactification.
\end{notation}

\subsection{Logarithmic pairs}

The following fundamental definitions of logarithmic geometry will be used in
the sequel.

\begin{defn}[Logarithmic pair]\label{def:everythinglog}
  A \emph{logarithmic pair} $(Z,\Delta)$ consists of a normal variety $Z$ and
  a reduced, but not necessarily irreducible Weil divisor $\Delta \subset Z$.
  A \emph{morphism of logarithmic pairs}, written as $\gamma: (\wtilde Z,
  \wtilde \Delta) \to (Z,\Delta)$, is a morphism $\gamma: \wtilde Z \to Z$
  such that $\gamma^{-1}(\Delta) = \wtilde \Delta$ set-theoretically.
\end{defn}

\begin{defn}[Snc pairs]\label{def:everythinglog2a}
  Let $(Z, \Delta)$ be a logarithmic pair, and $z \in Z$ a point. We say that
  $(Z, \Delta)$ is snc at $z$, if there exists a Zariski-open neighborhood $U$
  of $z$ such that $U$ is smooth and such that $\Delta \cap U$ has only simple
  normal crossings. The pair $(Z, \Delta)$ is snc if it is snc at all points.
  
  Given a logarithmic pair $(Z,\Delta)$, let $(Z,\Delta)_{\reg}$ be the
  maximal open set of $Z$ where $(Z,\Delta)$ is snc, and let
  $(Z,\Delta)_{\sing}$ be its complement, with the induced reduced subscheme
  structure.
\end{defn}

\begin{defn}[Log resolution]\label{def:everythinglog2b}
  A \emph{log resolution} of $(Z, \Delta)$ is a birational morphism of pairs
  $\pi: (\wtilde Z, \wtilde \Delta) \to (Z,\Delta)$ such that the
  $\pi$-exceptional set $\Exc(\pi)$ is of pure codimension one, such that
  $\bigl(\wtilde Z,\, \supp(\wtilde \Delta+ \Exc(\pi)) \bigr)$ is snc, and
  such that $\pi$ is isomorphic along $(Z, \Delta)_{\reg}$.
\end{defn}

If $(Z, \Delta)$ is a logarithmic pair, a log resolution is known to exist,
cf.~\cite{Kollar07}.

\subsection{Minimal model program}

We will use the definitions and apply the techniques of the minimal model
program frequently, sometimes without explicit references. On these occasions
the reader is referred to \cite{KM98} for background and details.

In particular, we will use the fact that the minimal model program asserts the
existence of \emph{extremal contractions} \cite[3.7, 3.31]{KM98} on
non-minimal varieties.  These extremal contractions come in three different
kinds: \emph{divisorial}, \emph{small}, and \emph{of fiber type}. The first
gives a birational morphism that contracts a divisor, the second leads to a
\emph{flip} \cite[2.8]{KM98}, and the third gives a fiber space.  Recall that
a fiber space $\pi : Y \to Z$ is called \emph{proper} if the general fiber $F$
is of dimension $0 < \dim F < \dim Y$. We will call an extremal contraction of
fiber type \emph{non-trivial} if the resulting fiber space is proper.
Finally, recall that extremal contractions of divisorial or fiber type have
relative Picard number one \cite[3.36]{KM98}.

Further note that since we are working in dimension at most $3$, we do not
need to appeal to the recent phenomenal advances in the Minimal Model Program
by Hacon-McKernan and Birkar-Cascini-Hacon-McKernan
\cite{MR2352762,BCHM06}. However, these results give us reasonable hope that
the methods here may extend to all dimensions.

\section{Singularities of the minimal model program}
\label{sec:singularities}

\subsection{Dlt singularities of index one}

If $(Z, \Delta)$ is an snc pair of dimension $\dim Z \leq 3$, the minimal
model program yields a birational map to a pair $(Z_\lambda, \Delta_\lambda)$,
where $Z_\lambda$ is $\mathbb Q$-factorial and $(Z_\lambda, \Delta_\lambda)$
is dlt ---see \cite[2.37]{KM98} for the definition of dlt. We remark for later
use that dlt pairs of index one are snc in codimension two.

\begin{lem}\label{lem:dltindexone}
  Let $(Z,\Delta)$ be a dlt pair of index one, i.e., a pair where $K_Z +
  \Delta$ is Cartier. Then
  \begin{equation}\label{eq:codimalongbdry}
    \codim_Z \bigl( (Z,\Delta)_{\sing} \cap \Delta \bigr) \geq 3.    
  \end{equation}
\end{lem}

\begin{subrem}
  It is important to note that $(Z, \Delta)$ has \emph{simple} normal
  crossings away from $(Z, \Delta)_{\sing}$, whereas having only normal
  crossings would give a much weaker result. This, for example, implies that
  the components of $\Delta$ are smooth in codimension $1$ which is not true
  for a boundary with only normal crossings, cf.~\cite[2.38]{KM98}.
\end{subrem}

\begin{proof}
  We will prove the statement by induction on the dimension.

  \smallskip

  \paragraph{\it Start of induction} First assume that $\dim Z=2$. Then by
  definition of dlt singularities, \cite[2.37]{KM98}, there exists a finite
  subset $T\subset Z$ such that $(Z, \Delta)_{\sing} \subseteq T$ and such
  that $Z$ is log terminal at the points of $T$, i.e., the discrepancy of any
  divisor $E$ that lies over $T$ is $a(E, Z, \Delta) > -1$.  But since $K_Z +
  \Delta$ is Cartier, this number must be an integer, so $a(E, Z, \Delta) \geq
  0$. This shows that $(Z, \Delta)$ is canonical at the points of
  $T$. Therefore it follows by \cite[4.5]{KM98} that $T \cap \Delta =
  \emptyset$. In particular,~\eqref{eq:codimalongbdry} holds.

  \smallskip

  \paragraph{\it Inductive step} Now let $Z$ be of arbitrary dimension, and
  let $H\subseteq Z$ be a general hyperplane section. Set $\Delta_H := \Delta
  \cap H$.  Since a Cartier divisor being smooth at a point implies that the
  ambient space is also smooth at that point, it follows, that for any $z\in
  H$, the pair $(H, \Delta_H)$ is snc at $z$ if and only if $(Z,\Delta)$ is
  snc at $z$. In other words, $(H, \Delta_H)_{\sing} = (Z, \Delta)_{\sing}
  \cap H$ and
  $$
  \codim_H \bigl( (H, \Delta_H)_{\sing} \cap \Delta_H \bigr) = \codim_Z \bigl(
  (Z, \Delta)_{\sing} \cap \Delta \bigr).
  $$
  Notice further that $(H, \Delta_H)$ is dlt of index one. The claim thus
  follows by induction.
\end{proof}

\subsection{Dlc singularities}
\label{sec:dlc}

Given an snc pair of Kodaira dimension zero, the minimal model program
terminates at a dlt pair $(Z, \Delta)$ where $\Delta$ is $\Q$-Cartier and $K_Z
+ \Delta$ is torsion. Much of the argumentation in Section~\ref{sec:k0} is
based on the following observation:
\begin{quotation}
  If $\Delta \not = \emptyset$, and $\varepsilon \in \Q^+$ sufficiently small,
  then $\bigl( Z, (1-\varepsilon)\Delta \bigr)$ is a dlt pair of Kodaira
  dimension $-\infty$. Therefore it admits at least one further extremal
  contraction.
\end{quotation}
Using the thinned down boundary to push the minimal model program further, we
end with a logarithmic pair $(Z', \Delta')$ that might no longer be dlt, but
still has manageable singularities.

\begin{defn}[Dlc singularities]
  A logarithmic pair $(Z', \Delta')$ is called \emph{dlc} if $(Z', \Delta')$
  is log canonical, $\Delta'$ is $\Q$-Cartier and for any sufficiently small
  positive number $\varepsilon \in \Q^+$, the pair $\bigl( Z',
  (1-\varepsilon)\Delta' \bigr)$ is dlt.
\end{defn}

Dlc singularities are of interest to us because sheaves of reflexive
differentials on dlc surface pairs enjoy good pull-back properties,
cf.~Theorem~\ref{thm:VZsheafextension2} below. For future reference, we recall
the relation between dlc and several other notions of singularity.

\begin{rem}[Relationship with other singularity classes]\label{rem:relation}
  By definition, a dlc pair $(Z, \Delta)$ is boundary-lc in the sense of
  \cite[Def.~3.6]{GKK08}. If $\dim Z = 2$ this implies that $(Z, \Delta)$ is
  finitely dominated by analytic snc pairs \cite[Lem.~3.9]{GKK08}.
\end{rem}

\subsection{Bogomolov-Sommese vanishing on singular spaces}

If $(Z, \Delta)$ is an snc pair, the well-known Bogomolov-Sommese vanishing
theorem asserts that for any number $1 \leq p \leq \dim Z$, any invertible
subsheaf $\sC \subseteq \Omega^{p}_Z(\log \Delta)$ has Kodaira-Iitaka
dimension at most $p$. See \cite[Sect.~6]{EV92} for a thorough
discussion. Much of the argumentation in this paper is based centrally on the
fact that similar results also hold for reflexive sheaves of differentials on
pairs with dlc, or more generally log canonical singularities.

The formulation of the general result we expect to be true is the following.

\begin{conjecture}[\protect{Bogomolov-Sommese vanishing for log canonical varieties}]
  \label{conjecture:Bvanishing}
  Let $(Z, \Delta)$ be a logarithmic pair and assume that $(Z, \Delta)$ is log
  canonical.  Let $\sA \subseteq \Omega^{[p]}_Z(\log \Delta)$ be any reflexive
  subsheaf of rank one. If $\sA$ is $\Q$-Cartier, then $\kappa(\sA) \leq p$.
\end{conjecture}

At this time Conjecture~\ref{conjecture:Bvanishing} has been verified with the
additional assumption $\dim Z\leq 3$ in \cite{GKK08}:

\begin{thm}[\protect{Bogomolov-Sommese vanishing for log canonical threefolds,
    \cite[Thm.~1.4]{GKK08}}]\label{thm:Bvanishing}
  Let $(Z, \Delta)$ be a logarithmic pair of dimension $\dim Z \leq 3$ and
  assume $(Z, \Delta)$ is log canonical. Let $\sA \subseteq
  \Omega^{[p]}_Z(\log \Delta)$ be any reflexive subsheaf of rank one. If $\sA$
  is $\Q$-Cartier, then $\kappa(\sA) \leq p$.  \qed
\end{thm}

\section{Global index-one covers for varieties of Kodaira dimension zero}
\label{sec:indexcoverk0}

In this section, we consider good minimal models of pairs with Kodaira
dimension $0$. We briefly recall the main properties of the global index-one
cover, as described in \cite[2.52]{KM98} or \cite[Sect.~3.6f]{Reid87}.

\begin{prop}\label{prop:index-cover}
  Let $(Z, \Delta)$ be a logarithmic pair.  Assume that the log canonical
  divisor $K_Z + \Delta$ is torsion (in particular, it is $\mathbb
  Q$-Cartier), i.e., assume that there exists a number $m \in \mathbb N^+$
  such that $\sO_Z\bigl(m\cdot(K_Z + \Delta)\bigr) \cong \sO_Z$.  Then there
  exists morphism of pairs $\eta: (Z', \Delta') \to (Z, \Delta)$, called the
  \emph{index-one cover}, with the following properties.
  \begin{enumerate-p}
  \item\ilabel{il:431} The morphism $\eta$ is finite. It is \'etale wherever
    $Z$ is smooth. In particular, $\eta$ is \'etale in codimension one.
  \item\ilabel{il:432} $K_{Z'}+ \Delta'$ is Cartier and $\O_{Z'}(K_{Z'}+
    \Delta') \simeq \O_{Z'}$.
  \item\ilabel{il:433} If $(Z, \Delta)$ is dlt, then $(Z', \Delta')$ is dlt as
    well.  If furthermore $z' \in Z'$ is a point where $(Z', \Delta')$ is not
    snc, then $(Z', \Delta')$ is canonical at $z'$.
  \end{enumerate-p}
\end{prop}
\begin{proof}
  Properties~\iref{il:431} and \iref{il:432} follow directly from the
  construction, cf.~\cite[2.50--53]{KM98}. To prove~\iref{il:433} assume for
  the remainder of the proof that $(Z, \Delta)$ dlt. We need to show that
  $(Z', \Delta')$ is dlt as well.  Observe that if $z' \in Z'$ is a point such
  that $(Z,\Delta)$ is snc at $\eta(z')$, then $(Z', \Delta')$ is snc at
  $z'$. The definition of dlt, together with the fact that discrepancies only
  increase under finite morphisms, \cite[5.20]{KM98}, then immediately yields
  the claim.
  
  Finally, if $z' \in Z'$ is any point where $(Z', \Delta')$ is not snc, then
  the discrepancy of any divisor $E$ that lies over $z'$ is $a(E, Z', \Delta')
  > -1$.  But since $K_{Z'} + \Delta'$ is Cartier, this number must be an
  integer, so $a(E, Z', \Delta') \geq 0$.  It follows that the pair $(Z',
  \Delta')$ is canonical at $z'$, hence \iref{il:433} is shown.
\end{proof}

\begin{cor}\label{cor:bdrynotempty}
  Under the conditions of Proposition~\ref{prop:index-cover}, if $\gamma:
  (\wtilde Z, \wtilde \Delta) \to (Z', \Delta')$ is any log resolution, then
  $\kappa(K_{\wtilde Z} + \wtilde \Delta) = 0$.
\end{cor}
\begin{proof}
  Since $(Z', \Delta')$ is canonical wherever it is not snc, the definition of
  canonical singularities, \cite[2.26, 2.34]{KM98}, implies that $K_{\wtilde
    Z} +\wtilde \Delta$ is represented by an effective, $\gamma$-exceptional
  divisor.
\end{proof}

\section{Viehweg-Zuo sheaves}
\label{sec:VZ}

\subsection{Definition of Viehweg-Zuo sheaves}

In the setup of Theorem~\ref{thm:mainresult2} and in a few other cases,
Viehweg and Zuo have shown in \cite[Thm.~1.4]{VZ02} that there exists a number
$n \gg 0$ and an invertible sheaf $\sA \subseteq \Sym^n \Omega^1_Y(\log D)$
whose Kodaira-Iitaka dimension is at least the variation of $f^\circ$, i.e.,
$\kappa(\sA) \geq \Var(f^\circ)$. The existence of this sheaf is a cornerstone
of our argumentation.

For technical reasons, it turns out to be more convenient to view $\sA$ as a
subsheaf of the tensor product, via the injection $\Sym^n \Omega^1_Y(\log D)
\into \bigl( \Omega^1_Y(\log D) \bigr)^{\otimes n}$.  It is also advantageous
to extend studying these sheaves on singular varieties and then it is natural
to allow rank one reflexive sheaves instead of restricting to line
bundles. These considerations give rise to the following definition.

\begin{defn}[Viehweg-Zuo sheaf]\label{def:VZ}
  Let $(Z,\Delta)$ be a logarithmic pair. A reflexive sheaf $\sA$ of rank $1$
  is called a \emph{Viehweg-Zuo sheaf} if there exists a number $n \in \mathbb
  N$ and an embedding $\sA \subseteq \bigl( \Omega^1_Z(\log \Delta)
  \bigr)^{[n]}$.
\end{defn}

\subsection{Pushing forward and pulling back }

We often need to compare Viehweg-Zuo sheaves on different birational models of
a pair. The following elementary statement shows that the push-forward of a
Viehweg-Zuo sheaf under a birational map of pairs is often again a Viehweg-Zuo
sheaf.

\begin{lem}[Push forward of Viehweg-Zuo sheaves]\label{lem:pushdownA}
  Let $(Z, \Delta)$ be a logarithmic pair and assume that there exists a
  Viehweg-Zuo sheaf $\sA \subseteq \bigl( \Omega^1_Z(\log \Delta)
  \bigr)^{[n]}$.  If $\lambda : Z \dasharrow Z'$ is a birational map whose
  inverse does not contract any divisor, $Z'$ is normal and $\Delta'$ is the
  (necessarily reduced) cycle-theoretic image of $\Delta$, then there exists a
  Viehweg-Zuo sheaf $\sA' \subseteq \bigl( \Omega^1_{Z'}(\log \Delta')
  \bigr)^{[n]}$ of Kodaira-Iitaka dimension $\kappa(\sA')\geq \kappa(\sA)$.
\end{lem}
\begin{proof}
  The assumption that $\lambda^{-1}$ does not contract any divisors and the
  normality of $Z'$ guarantee that $\lambda^{-1}: Z' \dasharrow Z$ is a
  well-defined embedding over an open subset $U \subseteq Z'$ whose complement
  has codimension $\codim_{Z'} (Z' \setminus U) \geq 2$, cf.~Zariski's main
  theorem \cite[V~5.2]{Ha77}.  In particular, $\Delta'\resto{U} = \bigl(
  \lambda^{-1}\resto{U} \bigr)^{-1}(\Delta)$. Let $\iota: U \into Z'$ denote
  the inclusion and set $\sA' := \iota_* \bigl( (\lambda^{-1}\resto{U})^{*}\sA
  \bigr)$. We obtain an inclusion of sheaves, $\sA' \subseteq \bigl(
  \Omega^1_{Z'}(\log \Delta') \bigr)^{[n]}$. By construction, we have that
  $h^0\bigl(Z',\, \sA'^{[m]}\bigr) \geq h^0(Z,\, \sA^{[m]})$ for all $m>0$,
  hence $\kappa(\sA') \geq \kappa(\sA)$.
\end{proof}

If $Z$ is a singular space with desingularization $\pi : \wtilde Z \to Z$, it
follows almost by definition that any differential $\sigma \in H^0 \bigl(Z,\,
\Omega^p_Z \bigr)$ pulls back to a differential $\pi^*(\sigma) \in H^0
\bigl(\wtilde Z,\, \Omega^p_{\wtilde Z} \bigr)$,
cf.~\cite[II~Prop.8.11]{Ha77}. However, if $\sigma$ is a reflexive
differential, i.e., if $\sigma \in H^0 \bigl(Z,\, \Omega^{[p]}_Z \bigr)$, it
is not all clear ---and generally false--- that $\pi^*(\sigma)$ can be
interpreted as a differential on $\wtilde Z$. Likewise, if $(Z, \Delta)$ is a
logarithmic pair with log resolution $\pi: (\wtilde Z, \wtilde \Delta) \to (Z,
\Delta)$ and $\sA \subseteq \bigl( \Omega^1_Z(\log \Delta) \bigr)^{[n]}$ a
Viehweg-Zuo sheaf, it is generally not possible to interpret the reflexive
pull-back $\pi^{[*]}(\sA)$ as a Viehweg-Zuo sheaf on $(\wtilde Z, \wtilde
\Delta)$. However, if the pair $(Z, \Delta)$ is log canonical, the extension
theorems for differential forms studied in \cite{GKK08} show that an
interpretation of $\pi^{[*]}(\sA)$ as a Viehweg-Zuo sheaf often exists.  The
following theorem is an immediate consequence of Remark~\ref{rem:relation} and
\cite[Thm.~8.1]{GKK08}. It summarizes the results of \cite{GKK08} that are
relevant for our line of argumentation.

\begin{thm}[\protect{Extension of Viehweg-Zuo sheaves, \cite[Thm.~8.1]{GKK08}}]\label{thm:VZsheafextension2} 
  Let $(Z,\Delta)$ be a dlc pair of dimension $\dim Z\leq 2$, and assume that
  there exists a Viehweg-Zuo sheaf $\sA$ with inclusion $\iota : \sA \into
  \bigl( \Omega^1_Z(\log \Delta) \bigr)^{[n]}$. If $\pi: (\wtilde Z, \wtilde
  \Delta) \to (Z, \Delta)$ is a log resolution, and
  $$
  E := \text{largest reduced divisor contained in }\pi^{-1}(\Delta) \cup
  \Exc(\pi),
  $$
  then there exists an invertible Viehweg-Zuo sheaf $\sC \subseteq \bigl(
  \Omega^1_{\wtilde Z}(\log E) \bigr)^{[n]}$ with the following property. For
  an arbitrary $m\in\bN$, the inclusion pulls back to give a sheaf morphism
  that factors through $\sC^{\otimes m}$,
  $$
  \bar \iota^{[m]} : \pi^{[*]} \left(\sA^{[m]} \right) \into \sC^{\otimes m}
  \subseteq \bigl( \Omega^1_{\wtilde Z}(\log E) \bigr)^{[m\cdot n]}.
  $$
  In particular, $\kappa(\sC) \geq \kappa(\sA)$. \qed
\end{thm}

\subsection{The reduction lemma}

Like regular differentials, logarithmic differentials come with a normal
bundle, and the corresponding restriction sequences, cf.~\cite[2.3]{EV92},
\cite[Lem.~2.13]{KK08} and the references there. Since Viehweg-Zuo sheaves
live in tensor products of the sheaf of differentials, this does not
immediately translate into a sequence for a given Viehweg-Zuo sheaf. This
makes the following lemma useful in the sequel.

\begin{lem}[Reduction lemma]\label{lem:reduction}
  Let $Z$ be a reduced irreducible variety, $\sE$, $\sF$, $\sG$, $\sH$ locally
  free sheaves, and $\sA$ a rank one torsion-free sheaf on $Z$.  Assume that
  there exists a short exact sequence
  \begin{equation}\label{eq:squce}
    0 \longrightarrow \sF \longrightarrow \sE \longrightarrow \sG \longrightarrow 0.    
  \end{equation}
  Then
  \begin{enumerate-p}
    \setcounter{enumi}{\value{equation}} 
  \item If there exists an inclusion $\sA\into \sE$, then either $\sA\into
    \sF$ or $\sA\into \sG$.
    \label{item:1}
  \item If for some $m\in\bN$, there exists an inclusion $\sA\into \sH\otimes
    \sE^{\otimes m}$, then there exists a $p\in \bN$, $0\leq p\leq m$ such
    that $\sA\into \sH\otimes\sF^{\otimes p}\otimes \sG^{\otimes m-p}$.
    \label{item:2}
  \item If for some $m\in\bN$, there exists an inclusion $\sA\into
    \sE^{\otimes m}$, and $\sF\simeq \sO_Z$ (respectively $\sG\simeq \sO_Z$),
    then there exists a $p\in \bN$, $0\leq p\leq m$ such that $\sA\into
    \sG^{\otimes p}$ (respectively $\sA\into \sF^{\otimes p}$).
    \label{item:3}
  \end{enumerate-p}
\end{lem}
\begin{proof}
  Suppose $\sA \into \sE$ and let $\sK=\ker[\sA\to\sG]\subseteq \sA$.  If
  $\sA\to\sG$ is injective at the general point of $Z$, then $\sK$ is a
  torsion sheaf and hence zero, so $\sA\into\sG$.  Since $\rk\sA=1$, if
  $\sA\to\sG$ is not injective at the general point, then it is zero.
  However, then $\factor \sA.\sK.\subseteq \sG$ is a torsion sheaf and hence
  zero, so $\sA\into \sF$.  This proves
  (\ref{lem:reduction}.\ref{item:1}). Taking $\sH=\sO_Z$ it is easy to see
  that (\ref{lem:reduction}.\ref{item:3}) is a special case of
  (\ref{lem:reduction}.\ref{item:2}). To prove
  (\ref{lem:reduction}.\ref{item:2}), we use induction.

  \smallskip

  \paragraph{\it Start of induction} If $m=1$, assertion
  (\ref{lem:reduction}.\ref{item:2}) follows from applying
  (\ref{lem:reduction}.\ref{item:1}) to the short exact sequence obtained by
  tensoring~\eqref{eq:squce} with $\sH$,
  $$
  0\to \sH\otimes\sF\to \sH\otimes\sE\to \sH\otimes\sG\to 0.
  $$
  Note that if $m=1$, then either $p = 0$ or $m-p = 0$.

  \smallskip

  \paragraph{\it Induction step} Now assume that the statement is true for all
  numbers $m'<m$. Consider the short exact sequence obtained by
  tensoring~\eqref{eq:squce} with $\sH \otimes \sE^{\otimes (m-1)}$,
  $$
  0\to\sH\otimes\sF\otimes \sE^{\otimes (m-1)} \to\sH\otimes\sE^{\otimes m}
  \to \sH\otimes\sG\otimes\sE^{\otimes (m-1)}\to 0.
  $$
  Applying (\ref{lem:reduction}.\ref{item:1}) for this short exact sequence
  yields that either $\sA\into (\sH\otimes\sF) \otimes \sE^{\otimes (m-1)}$ or
  $\sA \into (\sH\otimes\sG) \otimes \sE^{\otimes (m-1)}$. Setting $\sH' :=
  \sH\otimes\sF$ or $\sH' := \sH \otimes \sG$, respectively, and applying the
  induction hypothesis to the sequence
  $$
  0 \to \sH' \otimes \sF \otimes \sE^{\otimes (m-2)} \to \sH' \otimes
  \sE^{\otimes (m-1)} \to \sH' \otimes \sG \otimes \sE^{\otimes (m-2)} \to 0,
  $$
  we obtain a number $p\in \bN$, $0\leq p\leq m-1$ such that either $\sA\into
  (\sH\otimes\sF)\otimes\sF^{\otimes p}\otimes \sG^{\otimes m-1-p}$ or
  $\sA\into (\sH\otimes\sG)\otimes\sF^{\otimes p}\otimes \sG^{\otimes
    m-1-p}$. This proves (\ref{lem:reduction}.\ref{item:2}).
\end{proof}

\section{Viehweg-Zuo sheaves on minimal models}
\label{sec:VZ2}

The existence of a Viehweg-Zuo sheaf of positive Kodaira-Iitaka dimension
clearly has consequences for the geometry of the underlying space. In case the
underlying space is the end product of the minimal model program, we summarize
the two most important consequences below, when $\kappa = -\infty$ and $\kappa
= 0$.

\subsection{The Picard-number of minimal models with non-positive Kodaira dimension}

The following theorem will be used later to show that a given pair is a
Mori-Fano fiber space. This will turn out to be a crucial step in the proof of
our main results.

\begin{thm}\label{thm:b2thm}
  Let $(Z,\Delta)$ be a log canonical logarithmic pair where $Z$ is a
  projective $\Q$-factorial variety of dimension at most $3$. Assume that the
  following holds:
  \begin{enumerate-p}
  \item there exists a Viehweg-Zuo sheaf $\sA \subseteq \bigl( \Omega^1_Z(\log
    \Delta) \bigr)^{[n]}$ of positive Kodaira-Iitaka dimension, and
  \item\ilabel{il:c} the anti log canonical divisor $-(K_Z + \Delta)$ is
    nef.
  \end{enumerate-p}
  Then the Picard number of $Z$ is greater than one, $\rho(Z) > 1$.
\end{thm}
\begin{proof}
  We argue by contradiction and assume that $\rho(Z)=1$. Let $C \subseteq Z$
  be a general complete intersection curve. Since $C$ is general, it avoids
  the singular locus $(Z,\Delta)_{\sing}$.  By~\iref{il:c}, the restriction
  $\Omega^1_Z(\log \Delta)\resto {C}$ is a vector bundle of non-positive
  degree,
  \begin{equation}\label{eq:degneg}
    \deg \Omega^1_Z(\log \Delta)\resto C = (K_Z+\Delta).C \leq 0.
  \end{equation}
  
  We claim that the restriction $\Omega^1_Z(\log \Delta)\resto C$ is not
  anti-nef, i.e., that the dual vector bundle $\sT_Z(-\log \Delta)\resto C$ is
  not nef. Equivalently, we claim that $\Omega^1_Z(\log \Delta)\resto C$
  admits an invertible subsheaf of positive degree. Indeed, if
  $\Omega^1_Z(\log \Delta)\resto C$ was anti-nef, then none of its products
  $\bigl( \Omega^1_Z(\log \Delta)\resto C \bigr)^{[n]}$ could contain a
  subsheaf of positive degree.  However, since $C$ is general, the restriction
  of the Viehweg-Zuo sheaf to $C$ is a locally free subsheaf $\sA\resto C
  \subseteq \bigl( \Omega^1_Z(\log \Delta)\resto C \bigr)^{[n]}$ of positive
  Kodaira-Iitaka dimension, and hence of positive degree. This proves the
  claim.
  
  As a consequence of the claim and of Equation~\eqref{eq:degneg}, we obtain
  that $\Omega^{[1]}_Z(\log \Delta)$ is not semi-stable and if $\sB \subseteq
  \Omega^{[1]}_Z(\log \Delta)$ denotes the maximal destabilizing subsheaf,
  then its slope $\mu(\sB)$ is positive. The assumption that $\rho(Z)=1$ and
  the $\mathbb Q$-factoriality of $Z$ then guarantees that $\det \sB$ is a
  $\mathbb Q$-Cartier and $\mathbb Q$-ample sheaf of $p$-forms. Notice that by
  its choice the rank of $\sB$ has to be strictly less than the rank of
  $\Omega^{[1]}_Z(\log \Delta)$, hence $p < \dim Z$.  However, this leads to a
  contradiction. Because $\sB$ is $\bQ$-ample, it follows that $\kappa(\det
  \sB) = \dim Z$ violating the Bogomolov-Sommese Vanishing
  Theorem~\ref{thm:Bvanishing}.
\end{proof}

In the case when $Z$ is a surface, this theorem immediately gives a criterion
to guarantee that Viehweg-Zuo sheaves of positive Kodaira-Iitaka dimension
cannot exist.

\begin{cor}\label{cor:VZonFanoSurface}
  Let $(Z, \Delta)$ be a projective, logarithmic dlt surface pair where
  $-(K_Z+\Delta)$ is $\Q$-ample. If $\sA$ is any Viehweg-Zuo sheaf on $Z$,
  then its Kodaira-Iitaka dimension is non-positive, i.e.~$\kappa(\sA) \leq
  0$.
\end{cor}
\begin{proof}
  First recall from \cite[Prop.~4.11]{KM98} that $Z$ is $\Q$-Cartier.  The
  minimal model program then yields a morphism $\lambda: (Z,\Delta) \to
  (Z_\lambda ,\Delta_\lambda)$ to a $\Q$-Cartier model that does not admit any
  divisorial contractions. Note that the minimal model program for surfaces
  does not involve flips. Let $\sA_\lambda$ be the associated Viehweg-Zuo
  sheaf on $Z_\lambda$, as given by Lemma~\ref{lem:pushdownA}. It suffices to
  show that $\kappa(\sA_\lambda) \leq 0$.
  
  To this end, observe that $-(K_{Z_\lambda}+\Delta_\lambda)$ is still
  $\Q$-ample.  Theorem~\ref{thm:b2thm} and the Cone Theorem \cite[3.7]{KM98}
  then imply that there are at least two distinct contractions of fiber type,
  say $\pi_1 : Z_\lambda \to C_1$ and $\pi_2: Z_\lambda \to C_2$. If $F$ is a
  general fiber of $\pi_1$, then $F \cong \P^1$, the fiber $F$ is entirely
  contained inside the snc locus of $(Z_\lambda, \Delta_\lambda)$, and $F$
  intersects the boundary divisor $\Delta_\lambda$ transversely in no more
  than one point.  It follows from standard short exact sequences,
  \cite[Lem.~2.13]{KK08}, that
  $$
  \Omega^{[1]}_{Z_\lambda}(\log \Delta_\lambda)|_{F} \cong \sO_{\P^1} \oplus
  \sO_{\P^1}(a) \text{\quad with } a \leq 0.
  $$
  In particular, $\Omega^{[1]}_{Z_\lambda}(\log \Delta_\lambda)|_{F}$ is
  anti-nef, and $\sA_{\lambda}|_{F}$ is necessarily trivial. But the same
  holds for the restriction of $\sA_{\lambda}$ to general fibers of
  $\pi_2$. It follows that $\kappa(\sA_{\lambda}) \leq 0$, as claimed.
\end{proof}

\subsection{Viehweg-Zuo sheaves on good minimal models for varieties of logarithmic Kodaira dimension zero}

If $(Z,\Delta)$ is a good minimal model of Kodaira dimension zero, the
existence of a Viehweg-Zuo sheaf of positive Kodaira-Iitaka dimension implies
that $Z$ is uniruled.  This is shown next.

\begin{thm}\label{thm:uniruledness-of-minimal-models}
  Let $(Z,\Delta)$ be a logarithmic pair where $Z$ is projective. Assume that
  the following holds:
  \begin{enumerate}
  \item there exists a Viehweg-Zuo sheaf $\sA \subseteq \bigl( \Omega^1_Z(\log
    \Delta) \bigr)^{[n]}$ of positive Kodaira-Iitaka dimension,
  \item\ilabel{il:torsion} the log canonical divisor $K_Z + \Delta$ is
    $\mathbb Q$-Cartier and numerically trivial.
  \end{enumerate}
  Then $Z$ is uniruled.
\end{thm}
\begin{proof}
  We argue by contradiction and assume that $Z$ is \emph{not} uniruled. If
  $\pi: (\wtilde Z, \wtilde \Delta) \to (Z, \Delta)$ is any log resolution,
  this is equivalent to assuming that $K_{\wtilde Z}$ is pseudo-effective,
  \cite[cor.~0.3]{BDPP03}, see also \cite[sect.~11.4.C]{L04}. Again by
  \cite[thm.~0.2]{BDPP03}, this is in turn equivalent to the assumption that
  $K_{\wtilde Z}\cdot \wtilde C \geq 0$ for all moving curves $\wtilde C
  \subset \wtilde Z$.
  
  As a first step, we will show that the assumption implies that the (Weil)
  divisor $\Delta$ is zero. To this end, choose a polarization of $Z$ and
  consider a general complete intersection curve $C \subset Z$.  Because $C$
  is a complete intersection curve, it intersects the support of the effective
  divisor $\Delta$ non-trivially if the support is not empty.  By general
  choice, the curve $C$ is contained in the snc locus of $(Z, \Delta)$ and
  avoids the indeterminacy locus of $\pi^{-1}$.  Its preimage $\wtilde C :=
  \pi^{-1}(C)$ is then a moving curve in $\wtilde Z$ which intersects $\wtilde
  \Delta$ positively if and only if the Weil divisor $\Delta$ is not zero.
  But
  $$
  0 = \underbrace{(K_Z\vphantom{_{\wtilde Y}} + \Delta)}_{\text{num. triv.}}
  \cdot C = (K_{\wtilde Z} + \wtilde \Delta)\cdot \wtilde C =
  \underbrace{K_{\wtilde Z} \cdot \wtilde C}_{\geq 0, \text{ as $\wtilde C$ is
      moving}} + \underbrace{\wtilde \Delta\vphantom{_{\wtilde Y}} \cdot
    \wtilde C}_{\geq 0, \text{ as $\wtilde C\not \subseteq \wtilde \Delta$}},
  $$
  so $\wtilde \Delta \cdot \wtilde C = 0$, and then $\Delta = \emptyset$ as
  claimed.  Combined with Assumption~\iref{il:torsion}, this implies that the
  canonical divisor $K_Z$ is itself numerically trivial. The restrictions
  $\Omega^1_Z\resto{C}$ and $\sT_Z\resto{C}$ are locally free sheaves of
  degree zero, and so is the product $\bigl( \Omega^1_Z\resto{C}
  \bigr)^{\otimes n}$. On the other hand, the restriction $\sA\resto{C}
  \subseteq \bigl( \Omega^1_Z\resto{C} \bigr)^{\otimes n}$ has positive
  degree. In particular, $\bigl( \Omega^1_Z\resto{C} \bigr)^{[n]}$ is not
  semi-stable. Since products of semi-stable vector bundles are again
  semi-stable, \cite[Cor.~3.2.10]{HL97}, this implies that
  $\Omega^1_Z\resto{C}$ and $\sT_Z\resto{C}$ are likewise not semi-stable. In
  particular, the maximal destabilizing subsheaf of $\sT_{Z}\resto{C}$ is
  semi-stable and of positive degree, hence ample. In this setup, a variant
  \cite[Cor.~5]{KST07} of Miyaoka's uniruledness criterion
  \cite[Cor.~8.6]{Miy85} applies to give the uniruledness of $Z$, contrary to
  our assumption. For more details on this criterion see the survey
  \cite{KS06}.
\end{proof}

As a corollary, we obtain a criterion to guarantee that the boundary is not
empty. This will allow to apply the ideas described in Section~\ref{sec:dlc}
above.

\begin{cor}\label{cor:boundarynonempty}
  In the setup of Theorem~\ref{thm:uniruledness-of-minimal-models}, if $(Z,
  \Delta)$ is dlc, then the boundary divisor $\Delta$ is not empty.
\end{cor}
\begin{proof}
  We argue by contradiction and assume $\Delta = \emptyset$. By the definition
  of dlc, the pair $(Z, \emptyset)$ is then dlt.  Let $\eta: (Z', \emptyset)
  \to (Z, \emptyset)$ be the index-one-cover discussed in
  Proposition~\ref{prop:index-cover}.  Since $\eta$ is finite and étale in
  codimension one, there obviously exists an injection
  $$
  \eta^{[*]}(\sA) \subseteq \bigl( \Omega^1_{Z'} \bigr)^{[n]}.
  $$
  An application of Theorem~\ref{thm:uniruledness-of-minimal-models}, using
  the sheaf $\eta^{[*]}(\sA)$ as a Viehweg-Zuo sheaf on $(Z', \emptyset)$ then
  shows that $Z'$ is uniruled. If $\wtilde Z \to Z'$ is a resolution, then
  $Z'$ is likewise uniruled.  But Corollary~\ref{cor:bdrynotempty} would then
  assert that $\kappa(K_{\wtilde Z})=0$, in contradiction to uniruledness.
\end{proof}

\section{Unwinding families}
\label{sec:gluearama}

We will consider projective families $g: Y \to T$ where the base $T$ itself
admits a fibration $\varrho: T \to B$ such that $g$ is isotrivial on all
$\varrho$-fibers. It is of course generally false that $g$ is then the
pull-back of a family defined over $B$. We will, however, show in this section
that in some situations the family $g$ does become a pull-back after a
suitable base change. Most results in this section are probably known to
experts. We included full statements and proofs for the reader's convenience,
for lack of a suitable reference.

We use the following notation for fibered products that appear in our setup.

\begin{notation}
  Let $T$ be a scheme, $Y$ and $Z$ schemes over $T$ and $h:Y\to Z$ a
  $T$-morphism. If $t\in T$ is any point, let $Y_t$ and $Z_t$ denote the
  fibers of $Y$ and $Z$ over $t$. Furthermore, let $h_t$ denote the
  restriction of $h$ to $Y_t$. More generally, for any $T$-scheme $\wtilde T$,
  let
  $$
  h_{\wtilde T}: \underbrace{Y\times_T {\wtilde T}}_{=: Y_{\wtilde T}} \to
  \underbrace{Z\times_T {\wtilde T}}_{=: Z_{\wtilde T}}
  $$
  denote the pull-back of $h$ to $\wtilde T$. The situation is summarized in
  the following commutative diagram.
  $$
  \xymatrix{ Y_{\wtilde T} \ar[dr] \ar@/^1.25pc/[rrr] \ar[rr]_{h_{\wtilde T}}
    & & Z_{\wtilde T} \ar[dl] \ar@/^1.25pc/[rrr]|(.17)\hole & Y
    \ar[dr] \ar[rr]_{h}  && Z \ar[dl] \\
    & \wtilde T \ar[rrr] & & & T }
  $$
\end{notation}

The setup of the current section is then formulated as follows.

\begin{assumption}\label{ass:triplemor} 
  Throughout the present section, consider a sequence of morphisms between
  algebraic varieties,
  $$
  \xymatrix{ Y \ar[rr]^{g}_{\text{\tiny smooth, projective}} && {\ T\ }
    \ar[rr]^{\varrho}_{\overset{\text{\tiny smooth, rel.~dim.=1}}{\text{\tiny
          quasi-projective}}} && B, }
  $$
  where $g$ is a smooth projective family and $\varrho$ is smooth
  quasi-projective of relative dimension 1.  Assume further that for all $b
  \in B$, there exists a smooth variety $F_b$ such that for all $t \in T_b$,
  there exists an isomorphism $Y_t \simeq F_b$.
\end{assumption}

\subsection{Relative isomorphisms of families over the same base}

To start, recall the well-known fact that an isotrivial family of varieties of
general type over a curve becomes trivial after passing to an étale cover of
the base. As we are not aware of an adequate reference, we include a proof
here.

\begin{lem}\label{lem:dominant}
  Let $b\in B$ and assume that $\Aut(F_b)$ is finite.  Then the natural
  morphism $\iota: I=\Isom_{T_b}({Y_b}, {T_b}\times F_b)\to {T_b}$ is finite
  and étale.  Furthermore, pull-back to $I$ yields an isomorphism of
  $I$-schemes $Y_I \simeq I\times F_b$.
\end{lem}

\begin{proof}
  Consider the ${T_b}$-scheme
  $$
  H := \Hilb_{T_b} \bigl( {Y_b} \times_{T_b} ({T_b} \times F_b) \bigr) \simeq
  \Hilb_{T_b} \bigl( {Y_b} \times F_b \bigr).
  $$
  By Assumption~\ref{ass:triplemor}, $H_t\simeq \Hilb(F_b\times F_b)$ for all
  $t\in {T_b}$.  Similarly, $I_t\simeq \Aut(F_b)$ hence $I$ is
  one-dimensional, $\length(I_t)$ is constant on ${T_b}$ and $I\to T_b$ is
  dominant.  Since $I$ is open in $H$, the closure of $I$ in $H$, denoted by
  $H^I$, consists of a union of components of $H$. Therefore $H^I$ is also
  one-dimensional and since $H^I\to T$ is dominant, it is quasi-finite.
 
  Recall that $H\to {T_b}$ is projective, so $H^I\to {T_b}$ is also
  projective, hence finite.  Since $H\to {T_b}$ is flat, $\length(H^I_t)$ is
  constant.  Furthermore, $I\subseteq H^I$ is open, so $H^I_t=I_t$ and hence
  $\length(H^I_t)=\length(I_t)$ for a general $t\in {T_b}$.  However, we
  observed above that $\length(I_t)$ is also constant, so we must have that
  $\length(H^I_t)=\length(I_t)$ for all $t\in {T_b}$, and since $I\subseteq
  H^I$, this means that $I=H^I$ and $\iota:I\to {T_b}$ is finite and
  unramified, hence étale.

  In order to prove the global triviality of $Y_I$, consider $\Isom_I(Y_I,
  I\times F_b)$. Recall that taking $\Hilb$ and $\Isom$ commutes with base
  change, and so we obtain an isomorphism
  $$
  \Isom_I(Y_I, I\times F_b) \simeq I\times_{T_b} \Isom_{T_b}({Y_b},
  {T_b}\times F_b) \simeq I \times_{T_b} I.
  $$
  This scheme admits a natural section over $I$, namely its diagonal, which
  induces an $I$-isomorphism between $Y_I$ and $I \times F_b$.
\end{proof}

The preceding Lemma~\ref{lem:dominant} can be used to compare two families
whose associated moduli maps agree. In our setup any two such families become
globally isomorphic after base change.

\begin{lem}\label{lem:relglue}
  In addition to Assumption~\ref{ass:triplemor}, suppose that there exists
  another projective morphism, $Z\to T$, with the following property: for any
  $b\in B$ and any $t\in T_b$, we have $Y_t \simeq Z_t \simeq F_b$. Then
  \begin{enumerate-p}
  \item \ilabel{lem:54-1} there exists a surjective morphism $\tau: {\wtilde
      T}\to T$ such that the pull-back families of $Y$ and $Z$ to $\wtilde T$
    are isomorphic as $\wtilde T$-schemes, i.e., we have a commutative diagram
    as follows:
    $$
    \hskip 8em
    \xymatrix{ Y_{\wtilde T} \ar[dr] \ar@/^1.25pc/[rrr]
      \ar@{<->}[rr]_{\wtilde T-\text{isom.}}  & & Z_{\wtilde T}
      \ar[dl] \ar@/^1.25pc/[rrr]|(.17)\hole & Y
      \ar[dr]  && Z \ar[dl] \\
      & \wtilde T \ar[rrr]^\tau & & & T \ar[d]^\varrho\\
      & & & & B.}
    $$
  \end{enumerate-p}
  Furthermore, if for all $b\in B$, the group $\Aut(F_b)$ is finite, then
  $\wtilde T$ can be chosen such that the following holds. Let $\wtilde T'
  \subseteq \wtilde T$ be any irreducible component. Then
  \begin{enumerate-p}\addtocounter{enumi}{1}
  \item\ilabel{lem:54-2} $\tau$ is quasi-finite,
  \item\ilabel{lem:54-3} the image set $\tau (\wtilde T')$ is a union of
    $\varrho$-fibers, and
  \item\ilabel{lem:54-4} if $\wtilde T'$ dominates $B$, then there exists an
    open subset $B^\circ \subseteq (\varrho\circ\tau)(\wtilde T')$ such that
    $\tau\resto {\wtilde T'}$ is finite and étale over $B^\circ$.  More
    precisely, if we set $T^\circ := \varrho^{-1}(B^\circ)$ and $\wtilde
    T^\circ := \tau^{-1}(T^\circ) \cap \wtilde T'$, then the restriction
    $\tau\resto {\wtilde T^\circ} : \wtilde T^\circ \to T^\circ$ is finite and
    étale.
  \end{enumerate-p}
\end{lem}

\begin{subrem}
  In Lemma~\ref{lem:relglue} we do not claim that $\wtilde T$ is irreducible
  or connected.
\end{subrem}

\begin{proof}[Proof of Lemma~\ref{lem:relglue}]
  Set ${\wtilde T} := \Isom_T(Y,Z)$ and let $\tau:\wtilde T\to T$ be the
  natural morphism.  Again, taking $\Isom$ commutes with base change, and we
  have an isomorphism ${\wtilde T}\times_T{\wtilde T} \simeq \Isom_{{\wtilde
      T}}(Y_{{\wtilde T}}, Z_{{\wtilde T}}).  $ Similarly, for all $b\in B$,
  and for all $t\in T_b$, there is a natural one-to-one correspondence between
  ${{\wtilde T}}_t$ and $\Aut(F_b)$.  In particular, we obtain that $\tau$ is
  surjective. As before, observe that ${{\wtilde T}}\times_T {{\wtilde T}}$
  admits a natural section, the diagonal. This shows~\iref{lem:54-1}.
  
  If for all $b\in B$, $\Aut(F_b)$ is finite, then the restriction of $\tau$
  to any $\varrho$-fiber, $\tau_b:{{\wtilde T}}_b\to T_b$ is finite étale by
  Lemma~\ref{lem:dominant}.  This shows \iref{lem:54-2} and \iref{lem:54-3}.
  Furthermore, it implies that if $\wtilde T' \subseteq \wtilde T$ is a
  component that dominates $B$, neither the ramification locus of $\tau\resto
  {\wtilde T'}$ nor the locus where $\tau\resto {\wtilde T'}$ is not finite
  dominates $B$.

  Let $\what B\subseteq T$ be a multisection of $\varrho:T\to B$, i.e., a
  closed subvariety that dominates $B$ and is of equal dimension. In
  particular, the morphism $\varrho\resto{\what B}:\what B\to B$ is
  quasi-finite.  The scheme $\Isom_{\what B}(Y,Y)$ is quasi-finite and
  quasi-projective over $\what B$, hence over $B$ as well. Then there exists
  an open subset $B^\circ\subseteq B$ where $\length(\Isom_{\what
    B}(Y,Y))_{b}$ is constant for $b\in B$.  It is easy to see that
  \iref{lem:54-4} holds for $B^\circ$.
\end{proof}

\subsection{Families where $\varrho$ has a section}

In addition to Assumption~\ref{ass:triplemor} assume that the morphism
$\varrho$ admits a section $\sigma: B \to T$.  Using $\sigma:B\to T$, define
$Y_B := Y\times_T B$ and let $Z := Y_B \times_B T$ be the pull-back of $Y_B$
to $T$.  With these definitions, Lemma~\ref{lem:relglue} applies to the
families $Y \to T$ and $Z \to T$.  As a corollary, we will show below that in
this situation $\wtilde T$ contains a component $\wtilde T'$ such that the
pull-back family $Y_{\wtilde T'}$ comes from $B$.  Better still, the
restriction $\tau\resto {\wtilde T'}$ is ``relatively étale'' in the sense
that $\tau\resto {\wtilde T'}$ is étale and that $\varrho \circ \tau\resto
{\wtilde T'}$ has connected fibers.

\begin{cor}\label{cor:pulling-back-after-rel-etale}
  Under the conditions of Lemma~\ref{lem:relglue} assume that $\varrho$ admits
  a section $\sigma: B \to T$, and that $Z = Y_B \times_B T$.  Then there
  exists an irreducible component $\wtilde T' \subseteq \wtilde T$ such that
  \begin{enumerate}
  \item \ilabel{cor:56-1} $\wtilde T'$ surjects onto $B$, and
  \item \ilabel{cor:56-2} the restricted morphism $\wtilde\varrho := \varrho
    \circ \tau\resto {\wtilde T'} : \wtilde T' \to B$ has connected fibers.
  \end{enumerate}
\end{cor}
\begin{proof}
  It is clear from the construction that $Y_B \simeq Z_B$. This isomorphism
  corresponds to a morphism $\wtilde\sigma: B\to \Isom_T(Y,Z)=\wtilde T$. Let
  $\wtilde T' \subseteq \wtilde T$ be an irreducible component that contains
  the image of $\wtilde \sigma$.  Observe that $\wt\sigma$ is a section of
  $\wtilde\varrho : \wtilde T' \to B$ and that the existence of a section
  guarantees that $\wtilde\varrho$ is surjective and its fibers are connected.
\end{proof}

One particular setup where a section is known to exist is when $T$ is a
birationally ruled surface over $B$. The following will become important
later.

\begin{cor}\label{cor:family-push-forward}
  In addition to Assumption~\ref{ass:triplemor}, suppose that $B$ is a smooth
  curve and that the general $\varrho$-fiber is isomorphic to $\P^1$, $\mathbb
  A^1$ or $(\mathbb A^1)^*= \mathbb A^1\setminus \{0\}$. Then there exist
  non-empty Zariski open sets $B^\circ \subseteq B$, $T^\circ :=
  \varrho^{-1}(B^\circ)$ and a commutative diagram
  $$
  \xymatrix{%
    {\wtilde T^\circ} \ar@<-1pt>[r]^{\tau}_{\text{étale}}
    \ar@/_1mm/[rd]_{\text{conn.~fibers}} &
    \text{\vphantom{{$\wtilde T^\circ$}}}
    \ \ {T^\circ} \ar[d]^{\varrho} \\
    & \text{\phantom{\ }}{B^\circ} }
  $$
  such that 
  \begin{enumerate-p}
  \item the fibers of $\varrho\circ\tau$ are again isomorphic to $\P^1$,
    $\mathbb A^1$ or $(\mathbb A^1)^*$, respectively, and
  \item the pull-back family $Y_{\wtilde T^\circ}$ comes from $B^\circ$, i.e.,
    there exists a projective family $Z \to B^\circ$ and a $\wtilde
    T^\circ$-isomorphism
    $$
    Y_{\wtilde T^\circ} \simeq Z_{\wtilde T^\circ}.
    $$
  \end{enumerate-p}
\end{cor}
\begin{subrem}\label{srem:family-push-forward}
  If the general $\varrho$-fiber is isomorphic to $\P^1$ or $\mathbb A^1$, the
  morphism $\tau$ is necessarily an isomorphism. Shrinking $B^\circ$ further,
  if necessary, $\varrho: T^\circ \to B^\circ$ will then even be a trivial
  $\P^1$-- or $\mathbb A^1$--bundle, respectively.
\end{subrem}

\begin{proof}
  Shrinking $B$, if necessary, we may assume that all $\varrho$-fibers are
  isomorphic to $\P^1$, $\mathbb A^1$ or $(\mathbb A^1)^*$, and hence that $T$
  is smooth. Then it is always possible to find a relative smooth
  compactification of $T$, i.e. a smooth $B$-variety $\overline T \to B$ and a
  smooth divisor $D \subset T$ such that $\overline T \setminus D$ and $T$ are
  isomorphic $B$-schemes.
  
  By Tsen's theorem, \cite[p.~73]{Shaf94}, there exists a section $\sigma: B
  \to \overline T$.  In fact, there exists a positive dimensional family of
  sections, so that we may assume without loss of generality that $\sigma(B)$
  is not contained in $D$.
  
  Let $B^\circ \subseteq B$ be the open subset such that for all $b \in
  B^\circ$, $\overline T_b \simeq \P^1$, $T_b$ is isomorphic to $\P^1$,
  $\mathbb A^1$ or $(\mathbb A^1)^*$, respectively, and $\sigma(b) \not \in
  D$.  Using that any connected finite étale cover of $T_b$ is again
  isomorphic to $T_b$, and shrinking $B^\circ$ further,
  Corollary~\ref{cor:pulling-back-after-rel-etale} yields the claim.
\end{proof}

\begin{rem}
  Throughout the article we work over the field of complex numbers $\bC$, thus
  we kept that assumption here as well.  However, we would like to note that
  the results of this section work over an arbitrary algebraically closed base
  field $k$.
\end{rem}

\part{THE PROOFS OF THEOREMS~\ref*{thm:mainresult0}, \ref*{thm:mainresult2} AND \ref*{thm:mainresult3}}

\section{The case \texorpdfstring{$\kappa(Y^\circ)=-\infty$}{the Kodaira dimension is
    minus infinity}}
\label{sec:kinfty}

\subsection{Setup}

Let $f^\circ: X^\circ \to Y^\circ$ be a smooth projective family of varieties
with semi-ample canonical bundle, over a quasi-projective manifold $Y^\circ$
of dimension $\dim Y^\circ \leq 3$ and logarithmic Kodaira dimension
$\kappa(Y^\circ) = -\infty$.

Consider a smooth compactification $Y$ of $Y^\circ$ where $D := Y \setminus
Y^\circ$ is a divisor with simple normal crossings. Let $\lambda : Y
\dasharrow Y_\lambda$ be a sequence of extremal divisorial contractions and
flips given by the minimal model program, and let $D_\lambda \subset
Y_\lambda$ be the cycle-theoretic image of $D$.  We may assume that
$(Y_\lambda, D_\lambda)$ satisfies the following properties:

\begin{named_rem}{Properties of \texorpdfstring{$(Y_\lambda, D_\lambda)$}} \label{setup-Y_lambda}\ 
  \begin{enumerate-p}
  \item The variety $Y_\lambda$ is $\Q$-factorial, and $(Y_\lambda, D_\lambda)$ is a
    logarithmic dlt pair.
  \item The pair $(Y_\lambda, D_\lambda)$ does not admit a divisorial or small extremal contraction.
  \item As $\kappa(Y^\circ) = -\infty$, either \label{item:harom}
    \begin{itemize}
    \item $\rho(Y_\lambda)=1$ and $(Y_\lambda,D_\lambda)$ is $\Q$-Fano, or
    \item $\rho(Y_\lambda)>1$ and $(Y_\lambda,D_\lambda)$ admits a non-trivial extremal contraction of fiber type.
    \end{itemize}
  \end{enumerate-p}
\end{named_rem}

\subsection{Proof of Theorem~\ref*{thm:mainresult2}}
\label{sec:pf12kinfty}

To prove Theorem~\ref{thm:mainresult2}, assume that $f^\circ$ is a family of
canonically polarized varieties and that $f^\circ$ has positive variation,
$\Var(f^\circ) > 0$. By \cite[Thm.~1.4]{VZ02} and Lemma~\ref{lem:pushdownA},
this implies that there exists a Viehweg-Zuo sheaf $\sA_\lambda$ of positive
Kodaira-Iitaka dimension, $\kappa(\sA_\lambda) \geq \Var(f^\circ) > 0$ on
$(Y_\lambda,D_\lambda)$ .  Since $(Y_\lambda, D_\lambda)$ is $\Q$-factorial
and dlt, in particular log canonical, Theorem~\ref{thm:b2thm} implies that
$\rho(Y_\lambda)>1$.  Therefore, by (\ref{setup-Y_lambda}.\ref{item:harom}),
there exists an extremal contraction of fiber type $\pi: Y_\lambda \to C$. Let
$F \subset Y_\lambda$ be a general $\pi$-fiber, and $D_{\lambda,F} :=
D_\lambda \resto F$ the restriction of the boundary divisor.

We will now push the family $f^\circ$ down to $F$, to the maximum extent
possible. Since the inverse map $\lambda^{-1}$ does not contract any divisor,
we may use $\lambda^{-1}$ to pull the family $f^\circ: X^\circ \to Y^\circ$
back to obtain a smooth family of canonically polarized varieties,
$$
f_\lambda: X_\lambda \to Y_\lambda \setminus (D_\lambda \cup T),
\text{\quad where }\codim_{ Y_\lambda} T \geq 2.
$$
Let $f_{\lambda,F} := f_\lambda|_{F}$ be the restriction of this family to
$F$. To prove Theorem~\ref{thm:mainresult2} in our context, it suffices to
show that the family $f_{\lambda,F}$ is isotrivial. This will be carried out
next.

\subsubsection{Proof of Theorem~\ref*{thm:mainresult2} when $F$ is a curve}
\label{sec:K0curve}

If $F$ is a curve, it is entirely contained inside the snc locus of
$(Y_\lambda, D_\lambda)$ and does not intersect $T$. Furthermore, it follows
from the adjunction formula that $F \cong \P^1$ and that $D_{\lambda, F}$
contains no more than one point. In this situation, the isotriviality of
$f_{\lambda,F}$ is well-known, \cite[0.2]{Kovacs00a} and
\cite[Thm.~0.1]{Vie-Zuo01}. This shows that the variation $\Var(f^\circ)$
cannot be maximal and finishes the proof of Theorem~\ref{thm:mainresult2}.

\subsubsection{Proof of Theorem~\ref*{thm:mainresult2} when $F$ is a surface}

Again, we need to show that $f_{\lambda,F}$ is isotrivial. We argue by
contradiction and assume that this is \emph{not} not the case.  By general
choice of $F$, the pair $(F, D_{\lambda,F})$ is again dlt and
$$
\codim_F T_F = \codim_{ Y_\lambda} T \geq 2, \text{\quad where } T_F := T \cap
F.
$$

We claim that there exists a Viehweg-Zuo sheaf $\sB_\lambda$ on $(F,
D_{\lambda, F})$ which is of positive Kodaira-Iitaka dimension. In fact, an
embedded resolution of $D_{\lambda,F} \cup T_F\subseteq F$ provides an snc
pair $(\wtilde F, \wtilde D)$ and a proper morphism $\eta: \wtilde F \to F$
such that $\eta(\wtilde D) = D_{\lambda,F} \cup T_F$. The family
$f_{\lambda,F}$ pulls back to a family on $\wtilde F \setminus \wtilde D$, and
\cite[Thm.~1.4]{VZ02} asserts the existence of a Viehweg-Zuo sheaf $\sB$ on
$(\wtilde F, \wtilde D)$ with $\kappa(\sB) > 0$. The existence of a
Viehweg-Zuo sheaf $\sB_\lambda$ on $(F, D_{\lambda, F})$ with
$\kappa(\sB_\lambda) \geq \kappa(\sB) > 0$ then follows from
Lemma~\ref{lem:pushdownA}.

On the other hand, $-(K_F+ D_{\lambda,F})$ is $\Q$-ample because $\pi$ is an
extremal contraction of fiber type.  Corollary~\ref{cor:VZonFanoSurface} thus
asserts that $\kappa(\sB_\lambda) \leq 0$, a contradiction. This finishes the
proof of Theorem~\ref{thm:mainresult2} in case $\kappa(Y^\circ) =
-\infty$. \qed

\subsection{Proof of Theorem~\ref*{thm:mainresult3}}

We maintain the notation and assumptions made in Section~\ref{sec:pf12kinfty}
above and assume in addition that $Y$ is a surface. The minimal model map
$\lambda$ is then a morphism. As we have seen in Section~\ref{sec:K0curve} the
general fiber $F'$ of $\pi \circ \lambda$ is again a rational curve which
intersects the boundary in at most one point and that then the restriction of
the family $f^\circ$ to the fiber $F'\cap Y^\circ$ is necessarily isotrivial.
The detailed descriptions of $Y^\circ$ and of the moduli map in case
$\kappa(Y^\circ)=-\infty$ which are asserted in Theorem~\ref{thm:mainresult3}
then follow from Corollary~\ref{cor:family-push-forward} and
Remark~\ref{srem:family-push-forward}.  This finishes the proof of
Theorem~\ref{thm:mainresult3} in case $\kappa(Y^\circ) = -\infty$.  \qed

\subsection{Proof of Theorem~\ref*{thm:mainresult0}}

To prove Theorem~\ref{thm:mainresult0}, we argue by contradiction and assume
that both $\kappa(Y^\circ) = -\infty$ and that $\Var(f^\circ) = \dim Y^\circ$.
Lemma~\ref{lem:pushdownA} and \cite[Thm.~1.4]{VZ02} then give the existence of
a big Viehweg-Zuo sheaf $\sA_\lambda$ on $(Y_\lambda,D_\lambda)$. The
argumentation of Section~\ref{sec:pf12kinfty} applies verbatim and shows the
existence of a proper fibration of $\pi : Y_\lambda \to C$ such that the
induced family is isotrivial when restricted to the general $\pi$-fiber. That,
however, contradicts the assumption that the variation is
maximal. Theorem~\ref{thm:mainresult0} is thus shown in the case
$\kappa(Y^\circ) = -\infty$. \qed

\section{The case \texorpdfstring{$\kappa(Y^\circ)= 0$}{the Kodaira
    dimension is zero}}
\label{sec:k0}

\subsection{Setup}
\label{ssec:setup}

Let $f^\circ: X^\circ \to Y^\circ$ be a smooth projective family of varieties
with semi-ample canonical bundle, over a quasi-projective variety $Y^\circ$ of
dimension $\dim Y^\circ \leq 3$ and logarithmic Kodaira dimension
$\kappa(Y^\circ) = 0$. To prove Theorems~\ref{thm:mainresult0} and
\ref{thm:mainresult2} in this case, it suffices to show that $f^\circ$ is not
of maximal variation, and even isotrivial if its fibers are canonically
polarized. Since those families give rise to Viehweg-Zuo sheaves of positive
Kodaira-Iitaka dimension by \cite[Thm.~1.4]{VZ02},
Theorems~\ref{thm:mainresult0} and \ref{thm:mainresult2} immediately follow
from the following proposition.

\begin{prop}\label{prop:6.2}
  Let $(Z, \Delta)$ be a dlt logarithmic pair where $Z$ is a $\Q$-factorial
  variety of dimension $\dim Z \leq 3$ and $\kappa(K_Z+\Delta) = 0$. If $\sA$
  is any Viehweg-Zuo sheaf on $(Z,\Delta)$, then $\kappa(\sA) \leq 0$.
\end{prop}

Observe that once Theorem~\ref{thm:mainresult2} holds, the assertion of
Theorem~\ref{thm:mainresult3} is vacuous in our case. Accordingly, we do not
consider Theorem~\ref{thm:mainresult3} here.

We show Proposition~\ref{prop:6.2} in the remainder of the present
Section. The proof proceeds by induction on $\dim Z$. If $\dim Z = 1$, the
statement of Proposition~\ref{prop:6.2} is obvious. We will therefore assume
throughout the proof that $\dim Z > 1$, and that the following holds.

\begin{ihyp}\label{ihyp:p81}
  Proposition~\ref{prop:6.2} is already shown for all pairs $(Z', \Delta')$ of
  dimension $\dim Z' < \dim Z$.
\end{ihyp}

We argue by contradiction and assume the following.

\begin{iass}\label{ass:conk0ind}
  There exists a Viehweg-Zuo sheaf $\sA$ of positive Kodaira-Iitaka dimension
  $\kappa(\sA)>0$.
\end{iass}

We run the minimal model program and obtain a birational map $\lambda: Z
\dasharrow Z_\lambda$, where $Z_\lambda$ is $\Q$-factorial. If
$\Delta_\lambda$ is the cycle-theoretic image, the pair $(Z_\lambda,
\Delta_\lambda)$ is dlt, and $K_{Z_\lambda}+\Delta_\lambda$ is
semi-ample. Since $\kappa(K_{Z_\lambda}+\Delta_\lambda) = 0$, the divisor
$K_{Z_\lambda} + \Delta_\lambda$ is $\Q$-torsion, i.e.,
\begin{equation}\label{eq:torsion}
  \exists m \in \mathbb N \text{ such that } \sO_{Z_\lambda}
  \bigl( m(K_{Z_\lambda} + \Delta_\lambda)\bigr)
  \cong \sO_{Z_\lambda}.
\end{equation}
Lemma~\ref{lem:pushdownA} guarantees the existence of a Viehweg-Zuo sheaf
$\sA_\lambda$ on $(Z_\lambda,\Delta_\lambda)$ with
$\kappa(\sA_\lambda)>0$. Raising $\sA$ and $\sA_\lambda$ to a suitable
reflexive power, if necessary, we assume without loss of generality that
$\sA_\lambda$ is invertible and that $h^0(Z_\lambda,\, \sA_\lambda) > 0$.

\subsection{Outline of the proof}
\label{ssec:outline}

Since the proof of Proposition~\ref{prop:6.2} is slightly more complicated
than most other proofs here, we outline the main strategy for the convenience
of the reader.

The main idea is to apply induction, using a component of the boundary divisor
$\Delta_\lambda$. For that, we show in Section~\ref{ssec:nontriviality} that
$\sA_\lambda$ is not trivial on the boundary, and that there exists a
component $\Delta'_\lambda \subseteq \Delta_\lambda$ such that
$\kappa(\sA_\lambda\resto{\Delta'_\lambda}) > 0$. Passing to the index-one
cover, we will then in Section~\ref{ssec:8f} construct a Viehweg-Zuo sheaf of
positive Kodaira-Iitaka dimension on the associated boundary component and
verify that this component with its natural boundary satisfies all the
requirements of Proposition~\ref{prop:6.2}. This clearly contradicts the
Induction Hypothesis~\ref{ass:conk0ind} and finishes the proof.

In order to find $\Delta'_\lambda$ we need to analyze the geometry of
$Z_\lambda$ in more detail. For that, we will show in
Section~\ref{ssec:furthercontr} that the minimal model $Z_\lambda$ admits
further contractions if one is willing to modify the coefficients of the
boundary, compare the remarks in Section~\ref{sec:dlc}. A second application
of the minimal model program then brings us to a dlc logarithmic pair $(Z_\mu,
\Delta_\mu)$ that shares many of the good properties of $(Z_\lambda,
\Delta_\lambda)$. In addition, it will turn out in
Section~\ref{ssec:fiberspace} that $Z_\lambda$ has the structure of a Mori
fiber space. An analysis of the Viehweg-Zuo sheaf along the fibers will be
essential.

\subsection{Minimal models of $\boldsymbol{(Z_\lambda, (1-\varepsilon)\Delta_\lambda)}$}
\label{ssec:furthercontr}

As a first step in the program outlined in Section~\ref{ssec:outline}, we
claim that the boundary $\Delta_\lambda$ is not empty,
$\Delta_\lambda\neq\emptyset$.  In fact, using \eqref{eq:torsion} and the
existence of the Viehweg-Zuo sheaf $\sA_\lambda$, this follows immediately
from Corollary~\ref{cor:boundarynonempty}. In particular, \eqref{eq:torsion}
implies that $K_{Z_\lambda} \equiv - \Delta_\lambda$ and it follows that for
any rational number $0 < \varepsilon < 1$,
\begin{equation}\label{eq:kappa0infty}
  \kappa\bigl(K_{Z_\lambda}+(1-\varepsilon) \Delta_\lambda \bigr) =
  \kappa\bigl(\varepsilon K_{Z_\lambda} \bigr) =
  \kappa\bigl({Z_\lambda} \bigr) = -\infty.
\end{equation}

Now choose one $\varepsilon$ and run the log minimal model program for the dlt
pair $\bigl(Z_\lambda, (1-\varepsilon)\Delta_\lambda\bigr)$. This way one
obtains a birational map $\mu: Z_\lambda \ratmap Z_\mu$.  Let $\Delta_\mu$ be
the cycle-theoretic image of $\Delta_\lambda$.  The variety $Z_\mu$ is
$\Q$-factorial and the pair $\bigl(Z_\mu, (1-\varepsilon) \Delta_\mu \bigr)$
is then dlt.

\begin{claim}\label{claim:smuklt}
  The logarithmic pair $(Z_\mu, \Delta_\mu)$ is dlc.
\end{claim}
\begin{proof}
  By~\eqref{eq:torsion} some positive multiples of $K_{Z_\lambda}$ and
  $-\Delta_\lambda$ are numerically equivalent. For any two rational numbers
  $0 < \varepsilon', \varepsilon'' < 1$, the divisors $K_{Z_\lambda} +
  (1-\varepsilon') \Delta_\lambda$ and $K_{Z_\lambda} + (1-\varepsilon'')
  \Delta_\lambda$ are thus again numerically equivalent up to a positive
  rational multiple.
  
  The birational map $\mu$ is therefore a minimal model program for the pair
  $\bigl(Z_\lambda, (1-\varepsilon) \Delta_\lambda \bigr)$, independently of
  the number $\varepsilon$ chosen in its construction. It follows that
  $\bigl(Z_\mu, (1-\varepsilon') \Delta_\mu \bigr)$ has dlt singularities for
  all $0 < \varepsilon' < 1$, so $(Z_\mu, \Delta_\mu)$ is indeed dlc.
\end{proof}

\subsection{The fiber space structure of $\boldsymbol{Z_\mu}$}
\label{ssec:fiberspace}

Since the Kodaira-dimension of $\bigl(Z_\lambda,
(1-\varepsilon)\Delta_\lambda\bigr)$ is negative by~\eqref{eq:kappa0infty},
either $\rho(Z_\mu)=1$, or $\rho(Z_\mu)>1$ and the pair $\bigl( Z_\mu,
(1-\varepsilon) \Delta_\mu\bigr)$ admits an extremal contraction of fiber
type.  We apply Theorem~\ref{thm:b2thm} in order to show that the Picard
number cannot be one.

\begin{prop}\label{prop:mfs}
  The Picard number of $Z_\mu$ is not one. The pair $\bigl( Z_\mu,
  (1-\varepsilon) \Delta_\mu\bigr)$ therefore admits a non-trivial extremal
  contraction of fiber type, $\pi : Z_\mu \to W$.
\end{prop}
\begin{proof}
  As the birational map $\mu$ is a sequence of extremal divisorial
  contractions and flips, the inverse of $\mu$ does not contract any
  divisors. This has two consequences. First, the divisor
  $K_{Z_\mu}+\Delta_\mu$ is torsion, and $-(K_{Z_\mu}+\Delta_\mu)$ is nef. On
  the other hand, Lemma~\ref{lem:pushdownA} applies and shows the existence of
  a Viehweg-Zuo sheaf $\sA_\mu$ of positive Kodaira-Iitaka dimension. Since we
  have seen in Claim~\ref{claim:smuklt} that $(Z_\mu, \Delta_\mu)$ is dlc, in
  particular log canonical, and since we know that $Z_\mu$ is $\Q$-factorial,
  Theorem~\ref{thm:b2thm} then gives that $\rho(Z_\mu)>1$, as desired.
\end{proof}

Now let $F \subset Z_\mu$ be a general fiber of $\pi$, and set $\Delta_F :=
\Delta_\mu \cap F$. Since normality is preserved when passing to general
elements of base point free systems, \cite[Thm.~1.7.1]{BS95}, and since
discrepancies only increase, the logarithmic pair $(F, \Delta_F)$ is again
dlc.

\begin{rem}\label{rem:bdryFne}
  The adjunction formula gives that $K_F+\Delta_F$ is torsion. On the other
  hand, $\pi$ is an extremal contraction so $-\bigl(
  K_F+(1-\varepsilon)\Delta_F \bigr)$ is $\pi$-ample. It follows that the
  boundary divisor of $F$ cannot be empty, $\Delta_F \not = \emptyset$. It is
  not clear to us whether in general $F$ is necessarily $\Q$-factorial.
\end{rem}

\subsection{Non-triviality of $\boldsymbol{{\scr{A}}_\lambda\resto{\Delta_\lambda}}$} 
\label{ssec:nontriviality}

As in Section~\ref{ssec:setup}, Lemma~\ref{lem:pushdownA} guarantees the
existence of a Viehweg-Zuo sheaf $\sA_\mu$ on $(Z_\mu, \Delta_\mu)$ with
$\kappa(\sA_\mu) \geq \kappa(\sA_\lambda)>0$. Again, passing to a suitable
reflexive power, we can assume that $\sA_\mu$ is invertible and that
$h^0(Z_\mu,\, \sA_\mu) > 0$.

\begin{prop}\label{prop:c6}
  The restriction $\sA_\mu\resto{F}$ has Kodaira-Iitaka dimension zero,
  $\kappa(\sA_\mu\resto{F})=0$.
\end{prop}
\begin{proof}
  Consider the open set $F^\circ := (F, \Delta_F)_{\reg} \cap (Z_\mu,
  \Delta_\mu)_{\reg}$. The fiber $F$ being general, it is clear that
  $\codim_F(F \setminus F^\circ) \geq 2$. On $F^\circ$, the standard conormal
  sequence \cite[Lem.~2.13]{KK08} for logarithmic differentials then gives a
  short exact sequence of locally free sheaves, as follows,
  \begin{equation}\label{eq:conorm}
    0 \longrightarrow \underbrace{\pi^* \bigl(\Omega^1_W\bigr)\resto{F^\circ}}_{\text{trivial}}  \longrightarrow
    \Omega^1_{Z_\mu}(\log \Delta_\mu)\resto{F^\circ}  \longrightarrow \Omega^1_F(\log
    \Delta_F)\resto{F^\circ}  \longrightarrow 0.
  \end{equation}  
  By the definition of a ``Viehweg-Zuo sheaf'', there exists a number $n \in
  \mathbb N$ and an embedding $\sA_{\mu}\resto{F^\circ} \to \bigl(
  \Omega^1_{Z_\mu}(\log \Delta_\mu)\resto{F^\circ} \bigr)^{\otimes n}$. The
  first term in~\eqref{eq:conorm} being trivial, Lemma~\ref{lem:reduction}
  gives a number $m \leq n$ and an injection
  \begin{equation}\label{eq:embed}
    \sA_{\mu}\resto{F^\circ} \into \left( \Omega^1_F(\log \Delta_F)\resto{F^\circ}
    \right)^{\otimes m}.
  \end{equation}
  Recall that $\sA_\mu$ is invertible. Then by \eqref{eq:embed} we obtain an
  injection between the reflexive hulls $\sA_\mu\resto{F} \into \bigl(
  \Omega^1_F(\log \Delta_F) \bigr)^{[m]}$, i.e., we realize $\sA_\mu\resto{F}$
  as a Viehweg-Zuo sheaf on $(F, \Delta_F)$.
  
  The log canonical divisor $K_F+\Delta_F$ being torsion,
  Proposition~\ref{prop:c6} follows immediately if $F$ is a curve. We will
  thus assume for the remainder of the proof that $\dim F=2$.
  
  It remains to show that the Viehweg-Zuo sheaf $\sA_\mu\resto{F}$ on $(F,
  \Delta_F)$ has Kodaira-Iitaka dimension $\kappa(\sA_\mu\resto{F}) \leq
  0$. The fact that $\kappa(\sA_\mu) > 0$ will then imply that
  $\kappa(\sA_\mu\resto{F}) = 0$, as claimed. In order to do this, consider a
  log resolution $\psi: (\wtilde F, \wtilde \Delta_F) \to (F,
  \Delta_F)$. Setting
  $$
  E := \text{maximal reduced divisor in } \psi^{-1}(\Delta_F) \cup \Exc(\psi),
  $$
  it follows immediately from the definition of dlc that $K_{\wtilde F} + E$
  is represented by the sum of a torsion divisor and an effective,
  $\psi$-exceptional divisor. In particular, $\kappa(K_{\wtilde F} + E) = 0$,
  and Theorem~\ref{thm:VZsheafextension2} gives the existence of a Viehweg-Zuo
  sheaf $\sC$ on the snc pair $(\wtilde F, E)$ with $\kappa(\sC) \geq
  \kappa(\sA_\mu\resto{F})$. However, this contradicts the Induction
  Hypothesis~\ref{ihyp:p81}, which asserts that $\kappa(\sC) \leq 0$.
\end{proof}

\begin{cor}\label{cor:trivialA}
  The restriction $\sA_\mu\resto{F}$ is trivial, i.e., $\sA_\mu\resto{F} \cong
  \sO_F$.
\end{cor}
\begin{proof}
  Since $\sA_\mu$ is invertible and $h^0(Z_\mu,\, \sA_\mu) > 0$, there exists
  an effective Cartier divisor $D$ on $Z_\mu$ with $\sA_\mu \cong
  \sO_{Z_\mu}(D)$.  Decompose $D = D^h + D^v$, where $D^h$ consists of those
  components that dominate $W$, and $D^v$ of those components that do not. We
  need to show that $D^h = 0$. Again, we argue by contradiction and assume
  that $D^h$ is non-trivial.
  
  Recall that $\pi: Z_\mu \to W$ is a contraction of an extremal ray and that
  the relative Picard number $\rho(Z_\mu/W)$ is therefore one. The divisor
  $D^h$ is thus relatively ample, contradicting Proposition~\ref{prop:c6}.
\end{proof}

\begin{cor}\label{cor:nontriv}
  There exists a component $\Delta_{\lambda,1} \subseteq \Delta_\lambda$ such
  that $\kappa(\sA_\lambda\resto{\Delta_{\lambda,1}}) > 0$.
\end{cor}
\begin{proof}
  We have seen in Remark~\ref{rem:bdryFne} that $\Delta_F = \Delta_\mu \cap F$
  is not empty. So, there exists a component $\Delta_{\mu,1} \subseteq
  \Delta_\mu$ that intersects all $\pi$-fibers. Let $\Delta_{\lambda,1}
  \subseteq \Delta_\lambda$ be its strict transform. Since the birational map
  $\mu$ does not contract $\Delta_{\lambda,1}$, and since $\mu^{-1}$ does not
  contract any divisors, $\mu$ induces an isomorphism of open sets $U_\lambda
  \subseteq Z_\lambda$ and $U_\mu \subseteq Z_\mu$ such that
  $\Delta^{\circ}_{\lambda,1} := \Delta_{\lambda,1} \cap U_\lambda$ and
  $\Delta^{\circ}_{\mu,1} := \Delta_{\mu,1} \cap U_\mu$ are both non-empty.
  
  For an arbitrary $m \in \mathbb N$ we obtain a commutative diagram of linear
  maps,
  $$
  \xymatrix{ H^0\bigl(Z_\lambda,\, \sA_\lambda^{\otimes m} \bigr)
    \ar^(.45){\alpha_1}_(.45){\text{restr.}}[r] \ar[d]_{\mu_1} &
    H^0\bigl(\Delta_{\lambda,1},\, \sA_\lambda^{\otimes m}
    \resto{\Delta_{\lambda,1}}\bigr) \ar^{\alpha_2}_{\text{restr.}}[r] &
    H^0\bigl(\Delta^{\circ}_{\lambda,1},\, \sA_\lambda^{\otimes m}
    \resto{\Delta^{\circ}_{\lambda,1}}\bigr)  \ar[d]^{\mu_2} \\
    H^0\bigl(Z_\mu,\, \sA_\mu^{\otimes m} \bigr)
    \ar^(.45){\beta_1}_(.45){\text{restr.}}[r] & H^0\bigl(\Delta_{\mu,1},\,
    \sA_\mu^{\otimes m} \resto{\Delta_{\mu,1}}\bigr)
    \ar^{\beta_2}_{\text{restr.}}[r] & H^0\bigl(\Delta^{\circ}_{\mu,1},\,
    \sA_\mu^{\otimes m} \resto{\Delta^{\circ}_{\mu,1}}\bigr), }
  $$
  where the $\mu_i$, $i=1,2$ are the obvious push-forward morphisms coming
  from the construction of $\sA_\mu$ in Lemma~\ref{lem:pushdownA}. Since
  $\mu_1$ and $\beta_2$ are clearly injective, Corollary~\ref{cor:nontriv}
  will follow once we show that $\beta_1$ is injective as well. Now, let
  $\sigma \in H^0\bigl(Z_\mu,\, \sA_\mu^{\otimes m} \bigr)$ and assume that
  $\sigma$ is in the kernel of $\beta_1$.  By choice of $\Delta_{\mu,1}$, any
  general fiber $F$ intersects $\Delta_{\mu,1}$ in at least one point. The
  triviality of $\sA_\mu\resto{F}$ asserted in Corollary~\ref{cor:trivialA}
  then implies that $\sigma$ vanishes along $F$. The fiber $F$ being general,
  we obtain that $\sigma = 0$ on all of $Z_\mu$.  Corollary~\ref{cor:nontriv}
  follows.
\end{proof}

\subsection{Existence of pluri-forms on the boundary} 
\label{ssec:8f}

Now consider the index-one-cover $\gamma: (Z'_\lambda, \Delta'_\lambda) \to
(Z_\lambda, \Delta_\lambda)$, as described in
Proposition~\ref{prop:index-cover}. The pair $(Z'_\lambda, \Delta'_\lambda)$
is then dlt, the log canonical divisor is trivial, $\sO_{Z'_\lambda} \bigl(
K_{Z'_\lambda}+ \Delta'_\lambda \bigr) \cong \sO_{Z'_\lambda}$, and the
pull-back $\sA'_{\lambda} := \gamma^*(\sA_{\lambda})$ is an invertible
Viehweg-Zuo sheaf on $(Z'_\lambda, \Delta'_\lambda)$ with
$\kappa(\sA'_{\lambda}) > 0$. Better still, if $\Delta'_{\lambda,1} \subseteq
\gamma^{-1} \bigl( \Delta_{\lambda,1} \bigr)$ is any component,
Corollary~\ref{cor:nontriv} immediately implies that
$\kappa(\sA'_{\lambda}\resto{\Delta'_{\lambda,1}}) > 0$.

Now recall from Lemma~\ref{lem:dltindexone} that $(Z'_\lambda,
\Delta'_\lambda)$ is snc along the boundary away from a closed subset $W$ with
$\codim_Z (W\cap \Delta'_\lambda) \geq 3$. The divisor $\Delta'_\lambda$ is
therefore Cartier in codimension two and inversion of adjunction applies,
cf.~\cite[Sect.~5.4]{KM98}.  Setting
$$
\Delta''_{\lambda,1} := (\Delta'_{\lambda} - \Delta'_{\lambda,1})
\resto{\Delta'_{\lambda,1}},
$$
this yields the following:
\begin{enumerate-p}
\item the subvariety $\Delta'_{\lambda,1}$ is normal \cite[Cor.~5.52]{KM98}
  and
\item the pair $(\Delta'_{\lambda,1}, \Delta''_{\lambda,1})$ is again
  logarithmic and dlt \cite[Prop.~5.59]{KM98}.
\end{enumerate-p}

\begin{obs}
  It follows from the adjunction formula that the log canonical divisor
  $K_{\Delta'_{\lambda,1}} + \Delta''_{\lambda,1}$ is trivial.
\end{obs}

\begin{prop}
  The pair $(\Delta'_{\lambda,1}, \Delta''_{\lambda,1})$ admits a Viehweg-Zuo
  sheaf of positive Kodaira-Iitaka dimension.
\end{prop}
\begin{proof}
  Consider the standard conormal sequence for logarithmic differentials
  \cite[Lem.~2.13]{KK08} on the open subset $\Delta'^{\circ}_{\lambda,1} :=
  (\Delta'_{\lambda,1}, \Delta''_{\lambda,1})_{\reg}$,
  \begin{equation}\label{eq:diffsondelta}
    0 \longrightarrow \Omega^1_{\Delta'^{\circ}_{\lambda,1}}(\log \Delta''_{\lambda,1})
    \longrightarrow \Omega^1_{Z'_\lambda} (\log
    \Delta'_{\lambda})\resto{\Delta'^{\circ}_{\lambda,1}} \longrightarrow
    \sO_{\Delta'^{\circ}_{\lambda,1}} \longrightarrow 0.
  \end{equation}
  The last term in~\eqref{eq:diffsondelta} being trivial,
  Lemma~\ref{lem:reduction} gives a number $m \leq n$ and an injection
  $$
  \sA'_{\lambda}\resto{\Delta'^{\circ}_{\lambda,1}} \into \left(
    \Omega^1_{\Delta'^{\circ}_{\lambda,1}}(\log \Delta''_{\lambda,1})
  \right)^{\otimes m}.
  $$
  Using that $\sA'_{\lambda}$ is invertible and that
  $\codim_{\Delta'_{\lambda,1}} (W\cap \Delta'_{\lambda,1}) \geq 2$, we pass
  to reflexive hulls and realize $\sA'_{\lambda}$ as a Viehweg-Zuo sheaf on
  $\Delta'_{\lambda,1}$,

  \hfill $ \sA'_{\lambda}\resto{\Delta'_{\lambda,1}} \subseteq \left(
    \Omega^1_{\Delta'_{\lambda,1}}(\log \Delta''_{\lambda,1}) \right)^{[m]}.
  $\hfill
\end{proof}

\subsection{Completion of the proof}

Recall that we have seen that the pair $(\Delta'_{\lambda,1},
\Delta''_{\lambda,1})$ is dlt, has trivial log canonical class and admits a
Viehweg-Zuo sheaf of positive Kodaira-Iitaka dimension. Since $\dim
\Delta'_{\lambda,1} \leq 2$, being dlt implies that the variety
$\Delta'_{\lambda,1}$ is $\Q$-factorial, cf.~\cite[Prop.~4.11]{KM98}.  This
clearly contradicts the Induction Hypothesis~\ref{ihyp:p81}.
Assumption~\ref{ass:conk0ind} is therefore absurd. This finishes the proof of
Proposition~\ref{prop:6.2}. Consequently, Theorems~\ref{thm:mainresult0} and
\ref{thm:mainresult2} are shown in case $\kappa(Y^\circ) = 0$. \qed

\section{The case \texorpdfstring{$\kappa(Y^\circ) > 0$}{in case of
    positive Kodaira dimension}}
\label{sec:kpos}

\subsection{Setup}
\label{sec:sddljf1}

Let $f^\circ: X^\circ \to Y^\circ$ be a smooth projective family of varieties
with semi-ample canonical bundle over a quasi-projective variety $Y^\circ$ of
dimension $\dim Y^\circ \leq 3$ and logarithmic Kodaira dimension
$\kappa(Y^\circ) > 0$.

Again, let $Y$ be a compactification of $Y^\circ$ where $D := Y \setminus
Y^\circ$ is a divisor with simple normal crossings, and let $\lambda : (Y,D)
\dasharrow (Y_\lambda, D_\lambda)$ be the map to a minimal model. The divisor
$K_{Y_\lambda} + D_\lambda$ is then semi-ample by the log abundance theorem
\cite{MR2057020} and defines a map $\pi : Y_\lambda \to C$ with $\dim C =
\kappa(Y^\circ)$.

\subsection{Proof of Theorem~\ref*{thm:mainresult2}}
\label{sec:thm12ifkbig}

To prove Theorem~\ref{thm:mainresult2}, assume the $f^\circ$ is a family of
canonically polarized manifolds. We may also assume without loss of generality
that the family $f^\circ$ is not isotrivial and that $\kappa(Y^\circ) < \dim
Y$.  Blowing up $Y$ and pulling back the family, we obtain a diagram as
follows,
$$
\xymatrix{
  \wtilde X^\circ \ar[rrr]^{\wtilde f^\circ}_-{\text{family of canon. pol. var.}}
  \ar[d]_{\text{pull-back}} & & & (\wtilde Y, \wtilde D) \ar[d]^{\text{blow up}}
  \ar@/^2mm/[drrrr]^{\wtilde \pi} \\
  X^\circ \ar[rrr]^{f^\circ}_-{\text{family of canon. pol. var.}} & & &
  (Y,D) \ar@{-->}[rrr]^{\lambda}_-{\text{min. model program}} & & &
  (Y_\lambda, D_\lambda) \ar[r]_-{\pi} & C.  }
$$
If $\wtilde F \subset \wtilde Y$ is the general $\wtilde \pi$-fiber, recall
the standard fact that $\kappa(K_{\wtilde F} + \wtilde D|_{\wtilde F})=0$,
cf. \cite[sect.~11.6]{Iitaka82}. We saw in Section~\ref{sec:k0} that then the
family $\wtilde f^\circ$ must be isotrivial over $\wtilde F$. This shows that
the fibration $\pi$ factors the moduli map birationally, and proves
Theorem~\ref{thm:mainresult2} in case $\kappa(Y^\circ) > 0$. \qed

\subsection{Proof of Theorem~\ref*{thm:mainresult3}}

It remains to prove Theorem~\ref{thm:mainresult3} and give a detailed
description of the moduli map if $Y$ is a surface.

To this end, we maintain the notation and assumptions made in
Section~\ref{sec:sddljf1} above and assume in addition that $Y$ is a surface,
that $\Var(f^\circ) > 0$, and that $\kappa(Y^\circ)=1$. As there are no
flipping contractions in dimension two, $\lambda$ is a birational morphism,
and $K_{Y_\lambda} + D_\lambda$ is trivial on the general $\pi$-fiber
$F_\lambda \subset Y_\lambda$. In particular, one of the following holds:
\begin{itemize}
\item $F_\lambda$ is an elliptic curve and no component of $D_\lambda$
  dominates $C$, or
\item $F_\lambda$ is isomorphic to $\P^1$ and intersects $D_\lambda$ in
  exactly two points.
\end{itemize}
If the general fibers of $\pi$ are isomorphic to $(\mathbb A^1)^*$,
Corollary~\ref{cor:family-push-forward} gives the statement of
Theorem~\ref{thm:mainresult3}.

Otherwise, let $V \subseteq C$ be an open subset such that $\pi$ is a smooth
elliptic fibration over $V$. Let $\wtilde V\subset Y_\lambda$ be a general
hyperplane section. Restricting $V$ further if necessary we may assume that
$\wtilde V$ is \'etale over $V$.  Taking a base change to $\wtilde V$, we
obtain a section $\sigma: \wtilde V\to \wtilde U := U \times_V \wtilde V$.
Finally, set $\wtilde X := X \times_U \wtilde U$, and $Z := \wtilde V
\times_\sigma \wtilde X$. Shrinking $V$ further, if necessary, an application
of Lemma~\ref{lem:relglue} completes the proof. \qed

\subsection{Proof of Theorem~\ref*{thm:mainresult0}}

To prove Theorem~\ref{thm:mainresult0}, we argue by contradiction and assume
that $0 < \kappa(Y^\circ) < \dim Y^\circ$ and that $\Var(f^\circ) = \dim
Y^\circ$. The argumentation of Section~\ref{sec:thm12ifkbig} applies verbatim
and shows the existence of a proper fibration of $\wtilde \pi : \wtilde Y \to
C$ such that the family ${\wtilde f}^\circ$ is isotrivial when restricted to
the general $\wtilde \pi$-fiber. That, however, contradicts the assumption
that the variation is maximal. Theorem~\ref{thm:mainresult0} is thus shown in
case $\kappa(Y^\circ) > 0$. \qed

\providecommand{\bysame}{\leavevmode\hbox to3em{\hrulefill}\thinspace}
\providecommand{\MR}{\relax\ifhmode\unskip\space\fi MR}
\providecommand{\MRhref}[2]{%
  \href{http://www.ams.org/mathscinet-getitem?mr=#1}{#2}
}
\providecommand{\href}[2]{#2}

\end{document}